\documentclass[11pt]{amsart}
\usepackage{a4wide,color}

\usepackage{amsfonts, amsmath, amscd}
\usepackage[psamsfonts]{amssymb}
\usepackage{amssymb,textcomp,lmodern}
\usepackage{pb-diagram}
\usepackage[all,cmtip]{xy}

\newtheorem{theorem}{Theorem}[section]
\newtheorem{lemma}[theorem]{Lemma}
\newtheorem{corollary}[theorem]{Corollary}

\newtheorem{proposition}[theorem]{Proposition}

\newtheorem{example}[theorem]{Example}

\newtheorem{problem}[theorem]{Problem}

\newtheorem{claim}[theorem]{Claim}

\theoremstyle{definition}
\newtheorem{definition}[theorem]{Definition}
\newtheorem{remark}[theorem]{Remark}

\numberwithin{equation}{section}

\newcommand{\e}{\varepsilon}
\newcommand{\w}{\omega}
\newcommand{\Ra}{\Rightarrow}

\newcommand{\IR}{\mathbb{R}}
\newcommand{\IN}{\mathbb N}

\newcommand{\ff}{\mathbb{F}}

\newcommand{\Tau}{\mathcal T}

\newcommand{\IF}{\mathbb{F}}

\newcommand{\TTT}{\mathcal{T}}
\newcommand{\FF}{\mathcal{F}}

\newcommand{\M}{\mathcal{M}}

\newcommand{\supp}{\mathrm{supp}}

\newcommand{\spn}{\mathrm{span}}
\newcommand{\pr}{\mathrm{pr}}

\newcommand{\CC}{C_k}

\newcommand{\SM}{{\setminus}}

\input xy
\xyoption{all}

\title[Locally convex spaces with the strong Gelfand--Phillips property]{Locally convex spaces with \\ the strong Gelfand--Phillips property}
\author{Taras Banakh and Saak Gabriyelyan}

\address{Ivan Franko National University of Lviv (Ukraine) and Jan Kochanowski University in Kielce (Poland)}
\email{t.o.banakh@gmail.com}

\address{Department of Mathematics, Ben-Gurion University of the Negev, Beer-Sheva, P.O. 653, Israel}
\email{saak@math.bgu.ac.il}



\begin{document}

\begin{abstract}

We introduce the strong Gelfand--Phillips property for locally convex spaces and give several characterizations of this property. We characterize  the strong Gelfand--Phillips property among locally convex spaces admitting a stronger Banach space topology.
If $C_\TTT(X)$ is a space of continuous functions on a Tychonoff space $X$, endowed with a locally convex  topology $\TTT$ between the pointwise topology and the compact-open topology, then: (a) the space $C_\TTT(X)$ has the strong Gelfand--Phillips property iff $X$ contains a  compact countable subspace $K\subseteq X$ of finite scattered height such that for every functionally bounded set $B\subseteq X$ the complement $B\setminus K$ is finite, (b) the subspace $C^b_\TTT(X)$ of $C_\TTT(X)$  consisting of all bounded functions on $X$ has the strong Gelfand--Phillips property iff $X$ is a  compact countable space of finite scattered height.
\end{abstract}

\maketitle


\section{Introduction}


All topological spaces are assumed to be Tychonoff, all locally convex spaces are infinite-dimensional over the field $\mathbb{F}$ of real or complex numbers, and all operators between locally convex spaces are linear and continuous. For a locally convex space  $E$ (lcs for short), we denote by $E'$ the dual space of $E$.  For a bounded subset $B\subseteq E$ and a functional $\chi\in E'$,  we put $\|\chi\|_B:= \sup{(\{ |\chi(x)|:x\in B\cup\{0\})}$.

Let $E$ be a Banach space. A bounded subset $B$ of $E$ is called {\em limited\/} if every weak$^\ast$ null sequence $(\chi_n)_{n\in\w}$ in $E'$ converges uniformly on $B$, that is
$
\lim_n \|\chi_n\|_B =0.
$
A Banach space $E$ is called {\em Gelfand--Phillips} if every limited set in $E$ is precompact, i.e., has compact closure in $E$. 

Gelfand--Phillips Banach spaces were intensively studied by many authors, see for example \cite{Drewnowski,DrewEm,Schlumprecht-Ph,Schlumprecht-C} and more recent articles \cite{CGP,GhLe,Jos}. Since every Banach space is (isometrically) embedded in a $C(K)$-space, it is important to recognize Gelfand--Phillips spaces among Banach spaces of continuous functions in terms of the compact space $K$. Some sufficient conditions on compact spaces $K$ to have Gelfand--Phillips space $C(K)$  were obtained by Drewnowski \cite{Drewnowski}, Drewnowski and Emmanuele \cite{DrewEm}, and by Schlumprecht in \cite{Schlumprecht-Ph,Schlumprecht-C}.

{
In our recent article \cite{BG-GP-Banach}, we obtained several new characterizations of Banach spaces $E$ with the Gelfand--Phillips property using different idea: instead of {\em limited} sets in $E$, we consider  {\em bounded non-precompact} subsets of $E$. This approach  leads us to the following characterization of the Gelfand--Phillips property:
\begin{theorem} \label{t:GP-Banach-char}
A Banach space $E$ is Gelfand--Phillips if and only if for every bounded non-precompact set $B\subseteq E$, there is a weak$^\ast$ null sequence $(\chi_n)_{n\in\w}$ in $E'$ such that $\|\chi_n\|_B \not\to 0$.
\end{theorem}
The importance of this characterization of Gelfand--Phillips spaces is that it gives not only a new characterization of the Gelfand--Phillips property based on a much more handle class of bounded non-precompact sets than the class of limited sets, but it also allows to introduce and study the following {\em strong} version of that property.

\begin{definition}[\cite{BG-GP-Banach}] \label{def:Banach-strong-GP}
A Banach space $E$ is defined to have the {\em strong Gelfand--Phillips property} if $E$ admits a  weak$^\ast$ null-sequence $(\chi_n)_{n\in\w}$ in $E'$ such that $\|\chi_n\|_B \not\to 0$ for every  bounded non-precompact set $B\subseteq E$. In this case we shall say that $E$ is {\em strongly Gelfand--Phillips}.\qed
\end{definition}

It turns out that the class of strongly Gelfand--Phillips Banach spaces is rather narrow.

\begin{theorem}[\cite{BG-GP-Banach}] \label{t:Banach-strong-GP}
A Banach space $E$ is strongly Gelfand--Phillips if and only if it embeds into $c_0$.
\end{theorem}

In \cite{BG-GP-Banach} we also obtained a complete characterization of $C(K)$-Banach spaces with the strong Gelfand--Phillips property (for the definitions of scattered height and pseudocompactness, see Section \ref{sec:fs-sGP}).

\begin{theorem}[\cite{BG-GP-Banach}] \label{t:Banach-C(K)-strong-GP}
For an infinite pseudocompact space $K$ the following conditions are equivalent:
\begin{enumerate}
\item[{\rm(i)}] the Banach space $C(K)$ is strongly Gelfand--Phillips;
\item[{\rm(ii)}] the Banach space $C(K)$ is isomorphic to a subspace of $c_0$;
\item[{\rm(iii)}] the Banach space $C(K)$ is isomorphic to $c_0$;
\item[{\rm(iv)}] the space $K$ is countable, compact  and has finite scattered height.
\end{enumerate}
\end{theorem}

The aforementioned results and discussion naturally leads us to problem of extending the notion of a (strongly) Gelfand--Phillips space to the class of all locally convex space. It turns out that such an extension
cannot be done directly and needs some preliminary notions defined below.

Let $E$ be a locally convex space. We say that a subset $B\subseteq E$ is {\em barrel-bounded} if for any barrel $U\subseteq E$ there is an $n\in\w$ such that $B\subseteq nU$. A subset $A$ of $E$ is called {\em barrel-precompact} if for any barrel  $B\subseteq E$ there exists a finite set $F\subseteq E$ such that $A\subseteq F+B$. Clearly, each barrel-precompact set is barrel-bounded. It is easy to see that if $E$ is a Banach space, then a subset $B$ of $E$ is  barrel-bounded but not  barrel-precompact if and only if $B$ is a bounded non-precompact set of $E$.

Analogously to the case of Banach spaces we defined in \cite{BG-GP-lcs} the notion of limited sets and Gelfand--Phillips spaces as follows. A bounded subset $A$ of a locally convex space $E$ is called {\em limited} if every weak$^\ast$ null sequence $\{\chi_n\}_{n\in\w}$ in $E'$ converges uniformly on $A$, that is, $\lim_{n} \|\chi_n\|_A=0.$ A locally convex space $E$ is called a {\em Gelfand--Phillips space} or has the {\em Gelfand--Phillips property} if every limited subset of $E$ is barrel-precompact. In \cite{BG-GP-lcs} we obtain the following characterization of Gelfand--Phillips spaces which is similar to Theorem \ref{t:GP-Banach-char}.

\begin{theorem}[\cite{BG-GP-lcs}] \label{t:GP-lcs-char}
A locally convex space $E$ is Gelfand--Phillips if and only if for every barrel-bounded non-barrel-precompact set $B\subseteq E$, there is a weak$^\ast$ null sequence $(\chi_n)_{n\in\w}$ in $E'$ such that $\|\chi_n\|_B \not\to 0$.
\end{theorem}

The main goal of  this article is to define and study the strong Gelfand--Phillips property in the class of all locally convex spaces, in particular, in various classes of function spaces.

The discussion above shows that a Banach space $E$ has the  strong Gelfand--Phillips property if and only if it satisfies the following definition.

\begin{definition} \label{def:GP-strong}
A locally convex space $E$ is said to have the {\em strong  Gelfand--Phillips  property} ({\em strong $(GP)$ property}) if $E$ admits a  weak$^\ast$ null sequence $\{\chi_n\}_{n\in\w}$ in $E'$ such that $\|\chi_n\|_B \not\to 0$ for every barrel-bounded set $B\subseteq E$ which is not  barrel-precompact. Locally convex spaces with the strong Gelfan--Phillips property are called {\em strongly Gelfand--Phillips}.
\end{definition}

Our motivation to study of strongly Gelfand--Phillips spaces is explained also by the following. In \cite{BG-JNP} we introduce and study the class of Josefson--Nissenzweig locally convex spaces. For a locally convex space $E$, we denote by $\beta^\ast(E',E)$ the topology on $E'$ whose neighborhood base at zero consits of the polars of barrel-bounded subsets of $E$.

\begin{definition}[\cite{BG-JNP}] \label{def:JNP}
A locally convex space $E$ is said to have the {\em Jossefson--Nissenzweig property} (briefly, the JNP) if the identity map $\big(E',\sigma(E',E)\big) \to \big(E',\beta^\ast(E',E)\big)$ is  not sequentially continuous.\qed
\end{definition}
\noindent By the classical  Jossefson--Nissenzweig theorem every Banach space has the JNP, however, by \cite{BLV} (see also \cite{BG-JNP}), a  Fr\'{e}chet space has the JNP if and only if it is not Montel.

It is clear that for a locally convex space containing a barrel-bounded but not barrel-precompact set we have the implications
\[
\xymatrix@C=20pt{ \mbox{strong $(GP)$} \ar@{=>}[r] & \mbox{$(GP)$} \ar@{=>}[r] &\mbox{JNP} }
\]
and, by \cite{BG-GP-Banach}, none of these implications is reversible even in the class of Banach spaces.

Our study of the strongly  Gelfand--Phillips spaces is also motivated by the following very natural relationship between he notions introduced above in the simplest case of Banach spaces. Let $E$ be a Banach space. Then the Jossefson--Nissenzweig theorem states that there exists a weak$^\ast$ null sequence $\{\chi_n\}_{n\in\w}$ in $E'$ such that $\|\chi_n\|_{B_E}\not\to 0$. Taking into account that $B_E$ is trivially a barrell-bounded, non barrel-precompact subset of $E$ and keeping in mind the general case of locally convex spaces, one can naturally ask: When the sequence $\{\chi_n\}_{n\in\w}$ is {\em universal} by modulo  the Jossefson--Nissenzweig theorem in the sense that $\|\chi_n\|_{B}\not\to 0$ for {\em every} barrel-bounded, non-precompact subset $B$ of $E$?  Therefore we can consider the strong $(GP)$ property as a {\em universal} version of the Jossefson--Nissenzweig property.
On the other hand, we can also ask: When for a {\em chosen} barrel-bounded, non-barrel-precompact subset of $E$ there exists a weak$^\ast$ null sequence $\{\chi_n\}_{n\in\w}$ in $E'$ (now it depends also on $B$) such that  $\|\chi_n\|_{B}\not\to 0$? Therefore the  $(GP)$ property can be considered as a version of the Jossefson--Nissenzweig property which holds {\em everywhere}.

We can now describe the content of the paper.}
In Section \ref{sec:lcs-sGP} we characterize strongly  Gelfand--Phillips spaces, see Theorem \ref{t:char-sGP}. The study of Banach spaces endowed with the weak topology is one of the most important directions in Banach Space Theory, see for example \cite{Corson,edgar,fabian-10}. So locally convex spaces $E$ which admits a stronger Banach space topology is of independent interest. We shall say that a locally convex space $E$ is {\em $\beta$-Banach} if the space $E$ endowed with the topology whose neighborhood base at zero consists of barrels is a Banach space. In Theorem \ref{t:strong-GP} we characterize $\beta$-Banach spaces which have the strong  Gelfand--Phillips  property. As a corollary we obtain that for no Banach space $E$, the dual space $E'_{w^\ast}$ endowed with the weak$^\ast$ topology is strongly Gelfand--Phillips. We show that the class of strongly  Gelfand--Phillips spaces is closed under finite products, but even the countable power of a strongly  Gelfand--Phillips Banach space can fails to have the strong  Gelfand--Phillips property, see Proposition \ref{p:finite-prod-sGP} and Example \ref{exa:c-product-sGP}. In Example \ref{exa:b-feral-normed}  we show that the completion of a  strongly  Gelfand--Phillips  normed space can fail to be  strongly  Gelfand--Phillips, as well as a complete  strongly  Gelfand--Phillips space may contain a dense subspace without the  strong Gelfand--Phillips property, see Example \ref{exa:dense-subspace-sGP}.  As an application we present a characterization of Jossefson--Nissenzweig spaces in the terms of operators to spaces with the (strong) Gelfand--Phillips property, see Theorem \ref{t:JNP-characterization}.

For a Tychonoff space $X$, let $C(X)$ be the space of all continuous functions on $X$. Denote by $C^b(X)$ the subspace of $C(X)$ consisting of all bounded functions. Locally convex properties of the spaces $C(X)$ and $C^b(X)$ endowed with the topologies of uniform convergence on various families of sets in $X$ are widely studied in Functional Analysis, we refer the reader to the books \cite{Dales-Lau,kak,Schmets-76}. In Section \ref{sec:fs-sGP}  using general Theorem \ref{t:main3}, we obtain  complete characterizations of spaces $X$ for which the spaces $C(X)$ and $C^b(X)$ endowed with one of those topologies have  the strong  Gelfand--Phillips  property, see Theorem \ref{t:sGP-C} and Theorem \ref{t:Cb-sGP}, respectively.




\section{Locally convex spaces with the strong Gelfand--Phillips property} \label{sec:lcs-sGP}


In this section we characterize locally convex spaces with the strong Gelfand--Phillips property. First we recall some definitions.
Let $E$ be a locally convex space.  We denote by $E_w$ the space $E$ endowed with the weak topology and by $E'_{w^\ast}$ the dual space $E'$ endowed with the weak$^\ast$ topology. A subset $A$ of $E$ is called {\em precompact} if for any neighborhood $V$ of zero in $E$ there is a finite subset $F\subseteq E$ such that $A\subseteq F+V$. Let $U$ be a neighborhood of zero in $E$. A subset $A$ of $E$ is called {\em $U$-separated} if  $a-a'\not\in U$ for any distinct $a,a'\in A$. A subset of $E$ is {\em uniformly discrete} if it is  $U$-separated for some neighborhood $U$ of zero in $E$.
We shall use the following result which follows from  Theorem 5 of \cite{BGP}.
\begin{proposition} \label{p:lcs-precompact}
A subset $A$ of a locally convex space $E$ is precompact if and only if every uniformly discrete subset of $A$ is finite.
\end{proposition}
Let us recall also the following well-known description of precompact subsets of the Banach space $c_0$, where $e'_n$ is the $n$th coordinate functional of $c_0$.

\begin{proposition} \label{p:precompact-c0}
A subset $A$ of $c_0$ is precompact if and only if $\lim\limits_{n\to\infty}\|e'_n\|_A=0$.
\end{proposition}

For a locally compact space $X$, let
\[
C_0(X):=\{f\in C(X):\mbox{ the set }\{x\in X:|f(x)|\ge \e\}\mbox{ is compact for each } \e>0\}
\]
be the space of $\IF$-valued continuous functions on $X$  tending to zero at the infinity. The space $C_0(X)$ endowed with the sup-norm $\|f\|_{X} :=\sup\{ |f(x)|: x\in X\}$ is a Banach space. We denote by $C_p^0(X)$ the space $C_0(X)$ endowed with the subspace topology inherited from $C_p(X)$, where $C_p(X)$ denotes the space of $\mathbb{F}$-valued continuous functions on $X$ endowed with the pointwise topology.

Let $E$ be a locally convex space. Denote by $E_\beta$ the space $E$ endowed with the locally convex topology $\beta(E,E')$ whose neighborhood base at zero consists of barrels. Observe that a subset $A$ of $E$ is barrel-precompact if and only if $A$ is precompact in $E_\beta$, and $E$ is barrelled if and only if $E=E_\beta$.
Recall that a subset $A$ of $E$ is defined to be {\em barrel-precompact} if for any barrel  $B\subseteq E$ there exists a finite set $F\subseteq E$ such that $A\subseteq F+B$. In other words, $A$ is  barrel-precompact if and only if $A$ is precompact in the topology $\beta(E,E')$.

In the next theorem we characterize strongly Gelfand--Phillips spaces.
\begin{theorem} \label{t:char-sGP}
For a locally convex space $E$ the following assertions are equivalent:
\begin{enumerate}
\item[{\rm (i)}] $E$ has the strong Gelfand--Phillips property.
\item[{\rm (ii)}] There exists a null sequence $\{\chi_n\}_{n\in\w}$ in $E'_{w^\ast}$ such that for any infinite barrel-bounded, barrel-separated subset $D$ of $E$, there are  an infinite   subset $D_0 \subseteq D$ and $\delta>0$ such that
    \[
    \sup_{n\in \w}|\chi_n(x-x')|>\delta \; \mbox{ for every distinct $x,x'\in D_0$}.
    \]
\item[{\rm (iii)}] There is  a continuous operator $T:E\to C_p^0(\w)$ such that for any infinite barrel-bounded barrel-separated set $D$ in $E$, the set $T(D)$ is not  precompact in the Banach space $c_0$.
\end{enumerate}
\end{theorem}

\begin{proof}
(i)$\Rightarrow$(ii) Let $S=\{\chi_n\}_{n\in\w}$ be a null sequence in $E'_{w^\ast}$ witnessing the strong Gelfand--Phillips property. We show that it satisfies (ii). Indeed, let $D$ be an infinite barrel-bounded, barrel-separated subset of $E$. Then,  by Proposition \ref{p:lcs-precompact}, $D$ is not barrel-precompact. By the choice of $S$, we have $\|\chi_n\|_D \not\to 0$. Choose $\delta>0$ and a strictly increasing sequence $(n_k)_{k\in\w}\in\w$ such that $\|\chi_{n_k}\|_D >2\delta$ for every $k\in\w$. Now, choose arbitrarily $x_0\in D$ such that $|\chi_{n_0}(x_0)|>2\delta$ and set $k_0:=0$. Since $S$ is a null sequence in $E'_{w^\ast}$, by induction on $i>0$, we can find $k_i>k_{i-1}$ and $x_i\in D$ such that
\[
\big|\chi_{n_{k_j}}(x_i)\big|<\delta \;\; \mbox{ and }\;\; \big|\chi_{n_{k_j}}(x_j)\big|>2\delta \; \mbox{ for every } \; 0\leq i<j.
\]
Set $D_0:=\{x_i\}_{i\in\w}$. Then, by construction, we have
\[
\sup_{n\in \w}|\chi_n(x_i-x_j)|\geq  \big|\chi_{n_{k_j}}(x_i-x_j)\big|\geq \big|\chi_{n_{k_j}}(x_j)\big|-\big|\chi_{n_{k_j}}(x_i)\big|  >\delta
\]
for every distinct $i<j$ in $\w$, as desired.
\smallskip

(ii)$\Rightarrow$(iii) Let $S=\{\chi_n\}_{n\in\w}$ be a null sequence in $E'_{w^\ast}$ satisfying (ii). Then the operator
\[
T:E\to C_p^0(\w), \quad T(x):=\big( \chi_n(x)\big)_{n\in\w},
\]
is well-defined and continuous. Now, let $D$ be an  infinite barrel-bounded barrel-separated set in $E$. The choice of the sequence $S$ implies that there are  a sequence   $D_0=\{x_n\}_{n\in\w} \subseteq D$ and $\delta>0$ such that
\[
\sup_{n\in \w}|\chi_n(x_i-x_j)|>\delta \; \mbox{ for every distinct $i,j\in \w$}.
\]
But this means that $\|\chi_n\|_{D_0}\not\to 0$ and hence, by Proposition \ref{p:precompact-c0}, $T(D_0)$ is not precompact in $c_0$. Thus also $T(D)$ is not precompact in $c_0$, as desired.
\smallskip

(iii)$\Rightarrow$(i) Let $T:E\to C_p^0(\w)$ be  a continuous operator such that for any infinite barrel-bounded barrel-separated set $D$ in $E$, the set $T(D)$ is not  precompact in the Banach space $c_0$. For every $n\in\w$, set $\chi_n :=e'_n\circ T$, where $e'_n$ is the $n$th coordinate functional on $C_p^0(\w)$. Clearly, the sequence $S=\{\chi_n\}_{n\in\w}$ is null in $E'_{w^\ast}$. We show that $S$  witnesses the strong Gelfand--Phillips property of $E$. To this end, fix a barrel-bounded subset $B$ of $E$ which is not barrel-precompact. Then, by Proposition \ref{p:lcs-precompact} applied to the topology $\beta(E,E')$, $B$ contains an infinite barrel-separated subset $D$ of $E$. As a subset of $B$, the set $D$ is also barrel-bounded. Then the choice of the operator $T$ implies that $T(D)$ is not precompact in $c_0$. Finally, by  Proposition \ref{p:precompact-c0}, we obtain that
$\|\chi_n\|_B = \|e'_n\circ T\|_B \not\to 0$.
Thus $E$ has  the strong Gelfand--Phillips property.
%
%
\end{proof}


Let $E$ be  a Banach space. Recall that a subset $A$ of the dual space $E'$ of $E$ is called {\em norming} if there is a real constant $\lambda \geq 1$ such that
\begin{equation} \label{equ:norming}
\sup\{ |\chi(x)|: \chi \in A\cap B_{E'}\} \geq \tfrac{1}{\lambda} \| x\|
\end{equation}
for every $x\in E$, where $B_{E'}$ is the closed unit ball of $E'$.

We shall use repeatedly the following assertion.

\begin{lemma} \label{l:JNP-norming}
Let $E$ be a Banach space, $L$ be a norming subspace in $E'$, and let $\TTT$ be a weaker locally convex topology on $E$ such that $\sigma(E,L)\subseteq\TTT$. Then $(E,\TTT)_\beta=E$. In particular:
\begin{enumerate}
\item[{\rm (i)}] $\big(C_p^0(\w)\big)_{\!\beta}=c_0$; 
\item[{\rm (ii)}] if $X$ is a pseudocompact space and $\TTT$ is a locally convex vector topology on $C(X)$ such that $\Tau_p\subseteq \TTT\subseteq \TTT_{n}$, where $\TTT_{n}$ is the norm topology on $C(X)$, then
$
\big(C_\TTT(X)\big)_{\!\beta}=C(X).
$
\item[{\rm (iii)}] if $E=L'$ for some Banach space $L$, then $\big(E_{\w^\ast}\big)_{\!\beta}= E$.
\end{enumerate}
\end{lemma}

\begin{proof}
Choose $\lambda\geq 1$ such that  for $A=L$, the inequality (\ref{equ:norming}) holds for every $x\in E$. Define $S_L:= \{ \chi\in L: \| \chi\|=1\}$ and set $H:=(E,\TTT)$. 

Now, let $U$ be a barrel in $H$. Then $U$ is a barrel in $E$, and hence $E$ is a neighborhood of zero in the Banach space $E$.  Thus $\beta(H,H')$ is contained in the topology $\TTT_n$ generated by the norm of the Banach space $E$.

To prove that $\beta(H,H') \supseteq \TTT_{n}$, it suffices to show that the $\TTT$-closure $\overline{B}$ of the  closed unit ball  $B=\{x\in E:\|x\|\le 1\}$ of $E$ is contained in the ball $2\lambda B$ (because $\overline{B}$ is a  barrel in $H$, so $\tfrac{1}{2\lambda} \overline{B}\subseteq B$ is a $\beta(H,H')$-neighborhood of zero). To this end, for any $x\in E\setminus 2\lambda B$ choose $\chi_x \in S_L$ such that
\[
|\chi_x (x)| >\tfrac{1}{2\lambda} \| x\| >1.
\]
On the other hand, if $y\in B$, we have $|\chi_x(y)|\leq\|\chi_x\|\cdot\| y\| \leq 1$.
Since $\sigma(E,L)\leq\TTT$, each $\chi_x$ is $\TTT$-continuous and hence
\[
\overline{B} \subseteq \bigcap_{x\in E\SM 2\lambda B}\{y\in E:|\chi_x(y)|\le 1\} \subseteq \bigcap_{x\in E\SM 2\lambda B} E\SM\{x\} = 2\lambda B,
\]
as desired.
\smallskip

(i) To prove the  equality $\big(C_p^0(\w)\big)_{\!\beta}=c_0$, let $L:=\spn\big\{ e'_n:n\in\w\big\}$, where $e'_n$ is the  projection of $C_p^0(\w)$ onto the $n$th coordinate. Then $L$ is a norming subspace of $c_0$ and $C_p^0(\w)=(c_0,\sigma(c_0,L))$. As we proved above $\big(C_p^0(\w)\big)_{\!\beta}=c_0$.
\smallskip

(ii) It is clear that the subspace $L=C_p(X)'$ of $C(X)'$ is norming for the Banach space $C(X)$, and $\sigma(C(X),L)=\TTT_p \subseteq\TTT\subseteq \TTT_n$. Thus $\big(C_\TTT(X)\big)_{\!\beta}=C(X)$.
\smallskip

(iii) follows from the fact that $L$ is a norming subspace of $E'$.
\end{proof}

We define a locally convex space $E$ to be {\em $\beta$-Banach} if the space $E_\beta$ is topologically isomorphic to a Banach space. In particular, each Banach space is $\beta$-Banach.  Since each linear continuous functional on the space $E$ remains continuous in the topology of the space $E_\beta$, we can identify the dual space $E'$ of $E$ with a subspace of $(E_\beta)'$.

To prove Theorem \ref{t:strong-GP} we shall use  the following simple lemma.
\begin{lemma}\label{l:C0-c0-cont}
Each continuous operator $T:L\to C_p^0(\w)$ from a barrelled space $L$ remains continuous as an operator from $L$ to $c_0$.
\end{lemma}

\begin{proof}
Let $\{e'_n\}_{n\in\w}$ be the sequence of coordinate functionals on the Banach space $c_0$. The definition of the topology of the space $C_p^0(\w)$ ensures that each functional $e_n'$ remains continuous on the locally convex space $C_p^0(\w)$. Observe that the intersection $B:=\bigcap_{n\in\w}\{x\in C_p^0(\w):|e'_n(x)|\le 1\}$ coincides with the closed unit ball of the Banach space $c_0$. Since $B$ is a barrel also in $C^0_p(\w)$, the continuity of the operator $T$ implies that the set $T^{-1}(B)$ is a barrel in $L$. Since $L$ is barrelled, $T^{-1}(B)$ is a neighborhood of zero, which means that the operator $T:L\to c_0$ is continuous.
\end{proof}

Below we characterize $\beta$-Banach  locally convex spaces with the strong Gelfand--Phillips property. This result essentially generalizes Theorem \ref{t:Banach-strong-GP}.

\begin{theorem}\label{t:strong-GP}
A $\beta$-Banach locally convex space $E$ has the strong  Gelfand--Phillips  property if and only if  the space $E_\beta$ embeds into the Banach space $c_0$ and $E'$ is dense in the Banach space $(E_\beta)'$.
\end{theorem}

\begin{proof}
Assume that $E$ has the strong Gelfand--Phillips property. Then the dual space $E'_{w^\ast}$ contains a null sequence $\{\chi_n\}_{n\in\w}$ such that $\|\chi_n\|_B\not\to 0$ for any barrel-bounded subset $B\subseteq E$, which is not barrel-precompact. This implies that each barrel-bounded subset $B\subseteq Z$ in the closed linear subspace $Z=\bigcap_{n\in\w}\chi_n^{-1}(0)$ of $E$ is barrel-precompact. Then each bounded subset in the subspace $Z$ of the Banach space $E_\beta$ is precompact, which implies that the space $Z$ is finite-dimensional. Unifying the null sequence $\{\chi_n\}_{n\in\w}$ with a finite set of functionals separating points of the finite-dimensional subspace $Z$, we can assume that $Z=\{0\}$. In this case the   operator
\[
T:E\to C_p^0(\w), \quad T:x\mapsto\big(\chi_n(x)\big)_{n\in\w},
\]
is injective and continuous.

We claim that the operator $T:E_\beta\to c_0$ is an isomorphic topological embedding. Indeed, since the space $E_\beta$ is Banach and hence barrelled, Lemma \ref{l:C0-c0-cont} ensures that the operator $T:E_\beta\to c_0$ is continuous. By our assumption, the topology $\beta(E,E')$ of the space $E_\beta$ is generated by a (complete) norm $\|\cdot\|$.

Assuming that the operator $T:E_\beta\to c_0$ is not an embedding, we can find a sequence $\{x_n\}_{n\in\w}\subseteq E$ of elements of norm $\|x_n\|=1$ such that $\sup_{m\in\w}|\chi_m(x_n)|=\|T(x_n)\|\le \frac1{2^n}$ for every $n\in\w$.
We claim that the set $B=\{x_n\}_{n\in\w}$ is not precompact in the Banach space $E_\beta$. Indeed, in the opposite case, by the completeness of $E_\beta$, the sequence $\{x_n\}_{n\in\w}$ would contain a subsequence $\{x_{n_k}\}_{k\in\w}$ that converges in $E_\beta$ to some element $x_\infty\in E$ of norm $\|x_\infty\|=1$. The continuity of the operator $T$ ensures that $T(x_\infty)=\lim_{n\to\infty}T(x_n)=0$, which contradicts the injectivity of $T$. This contradiction shows that the set  $B$ is not precompact in $E_\beta$ and hence is not barrel-precompact in $E$. Now the choice of the sequence $\{\chi_n\}_{n\in\w}$ ensures that the sequence $\{\|\chi_n\|_B\}_{n\in\w}$ does not converge to zero. On the other hand, for every $\e>0$, we can find an $n\in\w$ such that $\frac1{2^n}<\e$. Since the sequence $\{\chi_k\}_{k\in\w}$ converges to zero in $E'_{w^\ast}$, there exists a natural number $m$ such $|\chi_k(x_i)|<\e$ for all $i\le n$ and $k\ge m$. Then for every $k\ge m$, we have
\[
\|\chi_k\|_B=\sup_{i\in\w}|\chi_k(x_i)|=\max\{\max_{i\le n}|\chi_k(x_i)|,\sup_{i>n}|\chi_k(x_i)|\}\le\max\{\e,\sup_{i>n}\tfrac1{2^i}\}\le\max\{\e,\tfrac1{2^n}\}=\e,
\]
which means that $\|\chi_k\|_B\to 0$. This contradiction shows that the operator $T:E_\beta\to c_0$ is a topological embedding.
\smallskip

To show that $E'$ is dense in the Banach space $(E_\beta)'$, we observe first that the adjoint map $T^\ast:(c_0)'=\ell_1 \to (E_\beta)'$ is surjective because $T$ is an embedding. Therefore it suffices to show that $E'$ contains all vectors $T^\ast(e'_n)$, where $\{ e'_n\}_{n\in\w}$ is the canonical basis in $\ell_1$. But this follows from the construction of the operator $T$ since $T^\ast(e'_n)=\chi_n$ for every $n\in\w$.
\smallskip

Conversely, assume now that the space $E_\beta$ embeds into the Banach space $c_0$ and $E'$ is dense in the Banach space $(E_\beta)'$.
We shall identify the space $E_\beta$ with a subspace of $c_0$. Let  $B=\{x\in E_\beta:\|x\|\le 1\}$ be the closed unit ball in the subspace $E_\beta$ of $c_0$.


For every $n\in\w$, let $\chi_n=e_n'{\restriction}_E$ be the restriction of the coordinate functional $e_n'\in c_0' =\ell_1$ to the subspace $E_\beta\subseteq c_0$. By our assumption, $\chi_n\in \overline{E'}\subseteq (E_\beta)'_\beta$, so there exists a functional $\mu_n\in E'$ such that
\begin{equation} \label{equ:JNP-01}
\|\mu_n-\chi_n\|:=\sup_{x\in B}|\mu_n(x)-\chi_n(x)|<\tfrac1{2^n}.
\end{equation}
 Since the sequence $\{e_n'\}_{n\in\w}$ is null in $(c_0)'_{w^\ast}$, the sequence $\{\chi_n\}_{n\in\w}$ is  $\sigma((E_\beta)',E_\beta)$-null in the dual space $(E_\beta)'$,  and
(\ref{equ:JNP-01}) implies that $\{\mu_n\}_{n\in\w}$ is $\sigma(E',E)$-null in $E'$.

Now take any barrel-bounded subset $P\subseteq E$ which is not barrel-precompact. Then there exists a barrel $D\subseteq E$ such that $P\not\subseteq F+D$ for any finite set $F\subseteq E$. Since $D$ is a barrel in the Banach space $E_\beta$, there exists $\e>0$ such that $\e B\subseteq D$. Then $P\nsubseteq F+\e B$ for any finite set $F\subseteq E$. This allows us to select inductively a sequence $\{x_n\}_{n\in\w}\subseteq P$ such that $\|x_m-x_n\|>\e$ for any $n<m$. Since the set $P\subseteq E_\beta\subseteq c_0$ is barrel-bounded, it is bounded in $\IF^\w$ and hence has compact closure in $\IF^\w$. Replacing the sequence $\{x_n\}_{n\in\w}$ by a suitable subsequence, we can assume that it converges to some element of $\IF^\w$.

We claim that there are two strictly increasing sequences $\{n_k\}_{k\in\w}$ and $\{m_k\}_{k\in\w}$ in $\w$ such that
\begin{equation} \label{equ:JNP-4}
|e'_{m_k}(x_{n_k})|>\tfrac{1}{3}\e.
\end{equation}
We proceed by induction on $k$. Since $\inf_{i\ne j}\|x_i-x_j\|\ge \e$, there exists a number $n_0\in\w$ such that $\|x_{n_0}\| >\tfrac{1}{3}\e$. By the definition of the norm $\|x_{n_0}\|$, there exists a number $m_0\in\w$ such that $|e_{m_0}'(x_{n_0})|>\tfrac{1}{3}\e$. Now, assume that for some $k\in\w$, we constructed strictly increasing sequences $\{n_i\}_{i<k}$ and $\{m_i\}_{i<k}$. Since the sequence $\{x_n\}_{n\in\w}$ converges in $\IF^\w$, there exists a number $l_k> n_{k-1}$ such that
\begin{equation} \label{equ:JNP-4-1}
\big|e'_p(x_i -x_j)|<\tfrac{2}{3}\e \; \mbox{ for all }\; i,j\ge l_k \;\mbox{ and }\; p\le m_{k-1}.
\end{equation}
As $\|x_{l_k}-x_{l_k+1}\|\ge\e$, (\ref{equ:JNP-4-1}) implies that there exists a natural number $m_k>m_{k-1}$ such that
$|e'_{m_k}(x_{l_k}-x_{l_k+1})|>\tfrac{2}{3}\e$.
Then, for some number $n_k\in\{l_k,l_k+1\}$, we get $|e'_{m_k}(x_{n_k})|>\frac13\e$.  Noting that $n_k\geq l_k>n_{k-1}$, we complete the inductive step.
\smallskip

Now, for every $k\in\w$, (\ref{equ:JNP-4}) implies
\[
\|\chi_{m_k}\|_P \ge |e'_{m_k}(x_{n_k})|>\tfrac{1}{3}\e.
\]
Observe also that, by  (\ref{equ:JNP-01}), we have
\[
|\mu_{m_k}(x_{n_k})-\chi_{m_k}(x_{n_k})|\le\|\mu_{m_k}-\chi_{m_k}\|\cdot\|P\|<\frac1{2^{m_k}}\|P\|,
\]
 where $\|P\|=\sup_{x\in P}\|x\|$ is finite as $P$ is bounded in $E_\beta\subseteq c_0$. Then (\ref{equ:JNP-4}) ensures that
\[
\|\mu_{m_k}\|_P\ge |\chi_{m_k}(x_{n_k})|-|\mu_{m_k}(x_{n_k})-\chi_{m_k}(x_{n_k})|>\tfrac13\e-\tfrac1{2^{m_k}}\cdot\|P\|,
\]
which implies that $\|\mu_n\|_P\not\to 0$, witnessing that the space $E$ is strongly Gelfand--Phillips.
\end{proof}

\begin{corollary} \label{c:sGP-C0p}
The space $C^0_p(\w)$ is strongly Gelfand--Phillips.
\end{corollary}

\begin{proof}
By Lemma \ref{l:JNP-norming}, we have $\big(C_p^0(\w)\big)_{\!\beta}=c_0$. To apply Theorem \ref{t:strong-GP} it remains to note that the dual space $C_p^0(\w)'$ coincides with the linear hull of the standard basis of the Banach space $\ell_1=(c_0)'=\big(C_p^0(\w)\big)_{\!\beta}'$.
\end{proof}

A weaker locally convex topology $\TTT$ on a Banach space $E$ is called {\em compatible} if $(E,\Tau)'=E$.

\begin{corollary} \label{c:weak-dual-Banach}
For a Banach space $E$ the following assertions are equivalent:
\begin{enumerate}
\item[{\rm(i)}] for every compatible locally convex topology $\TTT$ on $E$, the space $(E,\TTT)$ is strongly Gelfand--Phillips;
\item[{\rm(ii)}] $E_w$ is a strongly Gelfand--Phillips space;
\item[{\rm(iii)}] $E$ embeds into the Banach space $c_0$.
\end{enumerate}
\end{corollary}

\begin{proof}
The implication (i)$\Rightarrow$(ii) is trivial.
\smallskip

(ii)$\Rightarrow$(iii) By Lemma \ref{l:JNP-norming} we have $(E_w)_\beta= E$. Since $(E_w)'=E'$ the space $E_w$ is strongly Gelfand--Phillips, by Theorem \ref{t:strong-GP}.
\smallskip

(iii)$\Rightarrow$(i) By Lemma \ref{l:JNP-norming} we have $(E,\TTT)_\beta=E$. Since $(E,\TTT)'=E'$ the space $(E,\TTT)$ is a strongly Gelfand--Phillips space by Theorem \ref{t:strong-GP}.
\end{proof}

It follows from Corollary \ref{c:weak-dual-Banach} that there are Banach spaces which in the weak topology are strongly Gelfand--Phillips, for example $E=(c_0)_w$. It is natural to ask whether there exist a Banach space $E$ such that the weak$^\ast$ dual $E'_{w^\ast}$ is strongly Gelfand--Phillips. Below we answer this question in the negative. We recall that a Banach space $E$ is called {\em reflexive} if the canonical map $I: E\to E'', I(x)(\chi):=\chi(x)$, is surjective. Observe that since $I$ is an embedding and $E$ is complete, $E$ is reflexive if and only if $I(E)$ is dense in the bidual Banach space $E''$.

\begin{corollary} \label{c:weak*-dual-sGP}
For every Banach space $E$, the weak$^\ast$ dual $E'_{w^\ast}$ of $E$ does not have the strong Gelfand--Phillips property.
\end{corollary}

\begin{proof}
Suppose for a contradiction that $E'_{w^\ast}$ has the strong Gelfand--Phillips property. By Lemma \ref{l:JNP-norming} we have $(E'_{w^\ast})_\beta$ is the dual Banach space $E'$. Therefore $E'_{w^\ast}$ is $\beta$-Banach. By Theorem \ref{t:strong-GP}, $E'$ embeds into $c_0$ and $E=\big(E'_{w^\ast}\big)'$ is a dense subspace of the Banach bidual space $E''$. 
Hence the Banach space $E$ is reflexive. Therefore also $E'$ is reflexive. On the other hand, by Proposition 2.a.2 in \cite{LT}, each infinite-dimensional subspace of $c_0$ contains a copy of $c_0$, and hence the space $E'\subseteq c_0$ is not reflexive. This contradiction shows that $E'_{w^\ast}$ is not  strongly Gelfand--Phillips.
\end{proof}





\begin{example}\label{exa:Ostrovski}
There exists a separable locally convex space $E$ such that the space $E_\beta$ is isomorphic to the Banach space $c_0$ but $E$ is not strongly Gelfand--Phillips.
\end{example}

\begin{proof}
Let $Z$ be a closed, proper, and  norming subspace of $\ell_1$, and let $E:=\big(c_0,\sigma(c_0, Z)\big)$. Then, by Lemma \ref{l:JNP-norming}, $E_\beta$ is isomorphic to the space $c_0$. Since $Z$ is a closed, proper subspace of $\ell_1$, Theorem~\ref{t:strong-GP} implies that $E$ is not  strongly Gelfand--Phillips. Mikhail Ostrovskii suggested to take for $Z$ the linear subspace of $\ell_1$ consisting of  all sequences $z=(z_n)_{n\in\w}\in\ell_1$ such that $\sum_{n\in\omega}z_n=0$.
\end{proof}

Another example with the same properties can be found among function spaces $C_p(K)$.

\begin{example}
The locally space $c_p=\{x\in\IF^\w:\exists\lim_{n\to\infty}x(n)\}$ is not strongly Gelfand--Phillips but $(c_p)_\beta$ is isomorphic to the Banach space $c_0$.
\end{example}

\begin{proof}
Observe that the topology of $c_p$ is the topology on the Banach space $c$  defined by the norming subspace $L=\spn(e'_n)_{n\in\w}$ of $c'$, where $e'_n$ is the $n$th coordinate functional. Therefore, by Lemma \ref{l:JNP-norming}, $(c_p)_\beta =c$ and hence $c_p$ is $\beta$-Banach. It is well known that $c$ is isomorphic to $c_0$. Since $(c_p)'=L$ is not dense in $(c_p)'_\beta=c'$, Theorem \ref{t:strong-GP} implies that the space $c_p$ is not strongly Gelfand--Phillips.
\end{proof}

Now we consider some hereditary properties of the class of strongly  Gelfand--Phillips  spaces. We define a linear subspace $X$ of a locally convex space $E$ to be {\em $\beta$-embedded} if the identity inclusion $X_\beta\to E_\beta$ is a topological embedding. It is easy to see that $X$ is $\beta$-embedded in $E$ if and only if for any barrel $B\subseteq X$ there exists a barrel $D\subseteq E$ such that $D\cap X\subseteq B$.

\begin{proposition}\label{p:beta-emb}
A subspace $X$ of a locally convex space $E$ is $\beta$-embedded if one of the following conditions is satisfied:
\begin{enumerate}
\item[{\rm (i)}] $X$ is complemented in $E$;
\item[{\rm(ii)}] $X$ is barrelled;
\item[{\rm(iii)}] $X$ and $E$ are $\beta$-Banach and $X_\beta$ is closed in $E_\beta$.
\end{enumerate}
\end{proposition}

\begin{proof}
Given a barrel $B\subseteq X$, we should find a barrel $D\subseteq E$ such that $D\cap X\subseteq B$.
\smallskip

(i) If $X$ is complemented in $E$, then there exists a linear continuous operator $R:E\to X$ such that $R(x)=x$ for all $x\in X$. In this case the set $D=R^{-1}(B)$ is a barrel in $E$ with $D\cap X=B$.
\smallskip

(ii) If $X$ is barrelled, then the barrel $B$ is a neighborhood of zero. Since $X$ is a subspace of $E$, there exists a barrel neighborhood $D\subseteq E$ of zero such that $D\cap X\subseteq B$.
\smallskip

(iii) Assume that the spaces $X$ and $E$ are $\beta$-Banach and $X_\beta$ is closed in $E_\beta$. Then the identity inclusion $I:X_\beta\to E_\beta$ is a continuous injective operator between Banach spaces such that the image $I(X_\beta)$ is closed in $E_\beta$. By the Banach Open Mapping Principle, the operator $I:X_\beta\to E_\beta$ is a topological embedding.
\end{proof}

In the next proposition we give some sufficient conditions on a subspace of a  strong  Gelfand--Phillips  space to have the  strong Gelfand--Phillips property.

\begin{proposition} \label{p:sGP-hereditary}
Assume that a locally convex space $E$ is strongly Gelfand--Phillips. Then:
\begin{enumerate}
\item[{\rm (i)}] Every $\beta$-embedded subspace of $E$ is strongly Gelfand--Phillips.
\item[{\rm (ii)}] Every barrelled subspace of $E$ is strongly Gelfand--Phillips.
\item[{\rm (iii)}] If $E_\beta$ is barrelled (for example, $E$ is $\beta$-Banach), then $E_\beta$ is strongly Gelfand--Phillips.
\end{enumerate}
\end{proposition}

\begin{proof}
(i) Fix a weak$^\ast$ null sequence $\{\chi_n\}_{n\in\w}$ in $E'$ such that $\|\chi_n\|_A \not\to 0$ for every barrel-bounded set $A\subseteq E$ which is not  barrel-precompact. Let $X$ be a $\beta$-embedded subspace of $E$. Since the identity operator $I:X\to E$ is continuous, the sequence $(\chi_n\circ I)_{n\in\w}$ is null in the weak$^\ast$ topology on $X'$. We claim that the weak$^\ast$ null sequence $(\chi_n\circ I)_{n\in\w}$ witnesses the strong Gelfand--Phillips property of $X$. To this end, let $A$ be an infinite barrel-bounded but not barrel-precompact subset of $X$. Since $I$ is continuous the set $I(A)$ is barrel-bounded in $E$. To show that $I(A)$ is not barrel-precompact in $E$, we observe that, by Proposition \ref{p:lcs-precompact}, the set $A$ contains an infinite subset $D$ which is $B$-separated for some barrel $B\subseteq X$.  Since $X$  is $\beta$-embedded, there exists a barrel $\hat B\subseteq E$ such that $\hat B\cap X\subseteq B$. Observe that $I(D)$ is also $\hat B$-separated in $E$. Therefore,  by Proposition \ref{p:lcs-precompact},  $I(A)$ is not barrel-precompact in $E$. Now the choice of $(\chi_n)_{n\in\w}$ implies
$
\|\chi_n\circ I\|_A=\|\chi_n\|_{I(A)}\not\to 0.
$
Thus $X$ is strongly Gelfand--Phillips.
\vskip3pt

(ii) follows from (i) and  Proposition \ref{p:beta-emb}(ii).
\smallskip

(iii) Assume that the space $E_\beta$ is barrelled. By Theorem \ref{t:char-sGP}, there is  a continuous operator $T:E\to C_p^0(\w)$ such that for any infinite barrel-bounded barrel-separated set $D$ in $E$, the set $T(D)$ is not  precompact in the Banach space $c_0$. Set ${\bar T}:= T\circ i: E_\beta\to C_p^0(\w)$, where $i: E_\beta\to E$ is  the identity inclusion. Now take any infinite barrel-bounded barrel-separated subset $D\subseteq E_\beta$. Find a barrel $B\subseteq E_\beta$ such that $D$ is $B$-separated.  Since $E_\beta$ is barrelled, $B$ is a neighborhood of zero in $E_\beta$. By the definition of $E_\beta$, we can assume that $B$ is a barrel also in $E$. Hence $D$ is an infinite barrel-bounded barrel-separated subset of $E$. Then the image $T\circ i(D)=T(D)$  is not  precompact in  $c_0$. Thus, by Theorem  \ref{t:char-sGP}, $E_\beta$ is strongly Gelfand--Phillips.
\end{proof}

\begin{proposition} \label{p:finite-prod-sGP}
If $E$ and $L$ are locally convex spaces, then $E\times L$ is strongly Gelfand--Phillips if and only if so are $E$ and $L$.
\end{proposition}

\begin{proof}
Assume that $E\times L$ is a strongly Gelfand--Phillips space. By Proposition \ref{p:beta-emb}(i), $E$ and $L$ are $\beta$-embedded in $E\times L$ and therefore, by Proposition \ref{p:sGP-hereditary}(i), $E$ and $L$  are strongly Gelfand--Phillips.

Conversely, let $E$ and $L$ be strongly Gelfand--Phillips spaces. Choose weak$^\ast$ null sequences $(\chi_n)_{n\in\w} \subseteq E'$ and $(\eta_n)_{n\in\w} \subseteq L'$ witnessing the strong  Gelfand--Phillips property. Identify $E\times L$ with the direct sum $E\oplus L$ of the spaces $E$ and $L$.  Consider the null sequence $\{\xi_n\}_{n\in\w}\subseteq (E\oplus L)'_{w^\ast}$ defined by $\xi_{2n}:=\chi_n$ and $\xi_{2n+1}=\eta_n$ for every $n\in\w$. We claim that $\|\xi_n\|_B\not\to0$ for any  barrel-bounded set $B\subseteq E\times L$ which is not barrel-precompact. Given such a set $B$, observe that the projections $B_E$ and $B_L$ of $B$ onto $E$ and $L$, respectively, are barrel-bounded. Indeed, if $U$ is a barrel in $E$, then $U\times L$ is a barrel in $E\times L$. Therefore there is $a>0$ such that $B\subseteq a(U\times L)$. So $B_E \subseteq aU$ and hence $B_E$ is barrel-bounded.  Analogously one can prove that $B_L$ is barrel-bounded. We claim that at least one of the sets $B_E$ or $B_L$ is not barrel-precompact. To derive a contradiction, assume that $B_E$ and $B_L$ are barrel-precompact. Observe that for every barrel $D$ in $E\oplus L$, the intersection $D\cap E$ and $D\cap L$ are barrels in $E$ and $L$, respectively. By the barrel-precompactness of the sets $B_E$ and $B_L$, there are finite sets $F_E\subseteq E$ and $F_L\subseteq L$ such that $B_E\subseteq F_E+ \tfrac{1}{2} D$ and $B_L\subseteq F_L+\frac{1}{2}D$. Then $B\subseteq B_E+B_F\subseteq (F_E+F_L)+D$ and hence $B$ is barrel-precompact in $E\oplus L$, which contradicts the choice of $B$. This contradiction shows that one of the sets $B_E$ and $B_L$ is not barrel-precompact. Without loss of generality we assume that $B_E$ is not barrel-precompact. Then the choice of $(\chi_n)_{n\in\w}$ ensures that
$
\|\xi_{2n}\|_{B}=\| \chi_n\|_{B_E} \not\to 0.
$
Thus the sequence $(\xi_n)_{n\in\w}$ witnesses the strong Gelfand--Phillips  property of $E\times L$.
\end{proof}

\begin{remark}
In Example \ref{exa:c-product-sGP} we shall construct a strongly Gelfand--Phillips Banach space whose countable power is not strongly Gelfand--Phillips, and in Example~\ref{exa:closed-non-sGP} we construct a strongly Gelfand--Phillips locally convex space containing a closed subspace which is not strongly Gelfand--Phillips.\qed
\end{remark}

\begin{problem} Let $X$ be a closed subspace of a locally convex space $Y$.
\begin{itemize}
\item[{\rm (i)}] Is the quotient space $Y/X$ strongly Gelfand--Phillips if so is the space $Y$?
\item[{\rm (ii)}] Is $Y$ strongly Gelfand--Phillips if $X$ and $Y/X$ are strongly Gelfand--Phillips?
\end{itemize}
\end{problem}

{
We finish this section with the following characterization of Josefson--Nissenzweig spaces in the terms of operators to spaces with the (strong) Gelfand--Phillips property. We recall (see \cite{BG-GP-lcs}) that an operator $T:X\to Y$ between locally convex spaces is {\em $\beta$-to-$\beta$ precompact} if for any barrel-bounded set $B\subseteq X$ the image $T(B)$ is barrel-precompact in $Y$.

\begin{theorem} \label{t:JNP-characterization}
For a locally convex space $E$ 
the following conditions are equivalent:
\begin{enumerate}
\item[{\rm (i)}] $E$ has the Josefson-Nissenzweig property;
\item[{\rm (ii)}]  there exists a continuous operator $T:E\to C_p^0(\w)$, which is not $\beta$-to-$\beta$ precompact;
\item[{\rm (iii)}] for some locally convex space $Y$ with the strong Gelfand--Phillips property, there exists a continuous operator $T:E\to Y$, which is not $\beta$-to-$\beta$ precompact;
\item[{\rm (iv)}] for some locally convex space $Y$ with the Gelfand--Phillips  property, there exists a continuous operator $T:E\to Y$, which is not $\beta$-to-$\beta$ precompact.
\end{enumerate}
\end{theorem}

\begin{proof}
The equivalence (i)$\Leftrightarrow$(ii) is proved in \cite{BG-JNP}.

The implication (ii)$\Rightarrow$(iii) follows from Corollary \ref{c:sGP-C0p}, and (iii)$\Rightarrow$(iv) is trivial.
\smallskip

To prove that (iv)$\Ra$(i), assume that $E$ admits a continuous operator $T:E\to Y$ into a locally convex space $Y$ with the Gelfand--Phillips property such that $T$ is not $\beta$-to-$\beta$ precompact. Then for some barrel-bounded set $B\subseteq X$, the image $T(B)$ is not barrel-precompact in $Y$. Since $Y$ has the $(GP)$ property, there exists a null sequence $\{\mu_n\}_{n\in\w}\subseteq Y'_{w^\ast}$ such that  $\|\mu_n\|_{T(B)}\not\to 0$. For every $n\in\w$, consider the functional $\eta_n=\mu_n\circ T\in E'$ and observe that the sequence $\{\eta_n\}_{n\in\w}$ converges to zero in the topology $\sigma(E',E)$. However, since
\[
\| \eta_n\|_B =  \|\mu_n\|_{T(B)} \not\to 0 \quad \mbox{ as } \; n\to\infty,
\]
the sequence $\{ \eta_n\}_{n\in\w}$ does not converge to zero in the topology $\beta^\ast(E',E)$. This means that $E$ has the JNP.
\end{proof}

}


\section{Function spaces with the strong Gelfand--Phillips property} \label{sec:fs-sGP}


Let $X$ be a set, and let $f:X\to\IF$ be a function to the field $\ff$ of real or complex numbers. For a subset $A\subseteq X$ and $\e>0$, let
\[
\|f\|_A:=\sup(\{|f(x)|:x\in A\}\cup\{0\})\in [0,\infty],
\]
and if $\FF$ is a subfamily of $\ff^X$, we set
\[
[A;\e]_{\FF}:=\{f\in \FF:\|f\|_A\le\e\}.
\]
If the family $\FF$ is clear from the context, then we shall omit the subscript $\FF$ and write $[A;\e]$ instead of $[A;\e]_{\FF}$.
A family $\mathcal S$ of subsets of  $X$ is {\em directed} if for any sets $A,B\in\mathcal S$ the union $A\cup B$ is contained in some set $C\in\mathcal S$.

For a Tychonoff space $X$, we denote by $C(X)$ the space of all continuous functions $f:X\to \IF$ on $X$. A subset $A\subseteq X$ is called {\em functionally bounded} if $\|f\|_A<\infty$ for any continuous function $f\in C(X)$. A Tychonoff space $X$ is {\em pseudocompact} if $X$ is functionally bounded in $X$.

A Tychonoff space $X$ is defined to be a {\em $\mu$-space} if every functionally bounded subset of $X$ has compact closure in $X$. We denote by $\upsilon X$, $\mu X$ and $\beta X$ the Hewitt completion (=realcompactification), the Diedonn\'e completion and  the Stone-\v Cech compactification of $X$, respectively. It is known (\cite[8.5.8]{Eng}) that $X\subseteq \mu X\subseteq \upsilon X \subseteq \beta X$. Also it is known that all paracompact spaces and all realcompact spaces are Diedonn\'e complete and each Diedonn\'e complete space is a $\mu$-space, see \cite[8.5.13]{Eng}. On the other hand, each pseudocompact $\mu$-space is compact.

A topological space $X$ is {\em scattered} if each nonempty subspace of $X$ has an isolated point. For a topological space $X$, let $X^{(0)}:=X$ and let $X^{(1)}$ be the space of non-isolated points of $X$. For a non-zero ordinal $\alpha$, let $X^{(\alpha)}:=\bigcap_{\beta<\alpha}(X^{(\beta)})^{(1)}$. It is well known that a topological space $X$ is scattered if and only if $X^{(\alpha)}=\emptyset$ for some ordinal $\alpha$. The smallest ordinal $\alpha$ with $X^{(\alpha)}=\emptyset$ is called the {\em scattered height} of $X$.

For a Tychonoff space $X$, the space $C(X)$ carries many important {\em locally convex topologies}, i.e., topologies turning $C(X)$ into a locally convex space. For a locally convex topology $\TTT$ on $C(X)$, we denote by $C_\TTT(X)$ the space $C(X)$ endowed with the topology $\TTT$.

Each directed family $\mathcal S$ of functionally bounded sets in a Tychonoff space $X$ induces a locally convex topology $\TTT_{\mathcal{S}}$ on $C(X)$ whose neighborhood base at zero consists of the sets $[S;\e]$ where $S\in\mathcal{S}$ and $\e>0$. The topology $\TTT_{\mathcal{S}}$ is called {\em the topology of uniform convergence on sets of the family $\mathcal{S}$}. The topology $\TTT_{\mathcal{S}}$ is Hausdorff if and only if the union $\bigcup\mathcal{S}$ is dense in $X$.

If $\mathcal{S}$ is the family of all finite (resp. compact or functionally bounded) subsets of $X$, then the topology $\TTT_{\mathcal{S}}$ will be denoted by $\TTT_p$ (resp. $\TTT_k$ and  $\TTT_b$), and the function space $C_{\TTT_{\mathcal{S}}}(X)$ will be denoted by $C_p(X)$ (resp. $\CC(X)$ or $C_b(X)$).

\begin{lemma} \label{l:Ck=Cb}
A Tychonoff space $X$ is a $\mu$-space if and only if $\CC(X)=C_b(X)$.
\end{lemma}

\begin{proof}
If $X$ is a $\mu$-space, then each closed functionally bounded set in $X$ is compact, which implies the equality $C_b(X)=\CC(X)$. Now assume conversely that $C_b(X)=\CC(X)$. Given any closed functionally bounded set $B\subseteq X$, consider the neighborhood of zero $[B;1]$ in $C_b(X)$. Since $C_b(X)=\CC(X)$, there exists a compact set $K\subseteq X$ and $\e>0$ such that $[K;\e]\subseteq[B;1]$. We claim that $B\subseteq K$. In the opposite case, by the Tychonoff property of $X$, we could find a continuous function $f:X\to\IR$ such that $f(K)\subseteq\{0\}$ and $f(x)>1$ for some point $x\in B\setminus K$. Then $f\in [K;\e]\subseteq [B;1]$ and hence $|f(x)|\le 1$, which contradicts the choice of $f$. This contradiction shows that $B\subseteq K$ and hence  $B$ is compact.
\end{proof}

Let $X$ be a dense subspace of a Tychonoff space $M$ (for example, $M=\mu X$, $\upsilon X$ or $\beta X$). Then the union $\bigcup\mathcal{S}$ of the directed  family of all finite (resp. compact) subsets of $X$ is dense in $M$. Therefore $\mathcal{S}$  naturally defines a Hausdorff locally convex vector topology on the space $C(M)$, which will be denoted by $\TTT_{p{\restriction}X}$ (resp. $\TTT_{k{\restriction}X}$). For simplicity of notations we shall denote the function spaces $C_{\TTT_{p{\restriction}X}}(M)$ and $C_{\TTT_{k{\restriction}X}}(M)$ by $C_{p{\restriction}X}(M)$ and $C_{k{\restriction}X}(M)$,  respectively.

For a Tychonoff space $X$ and a linear functional $\mu\in C(X)'$, the {\em support} $\supp(\mu)$ of $\mu$ is the set of all points $x\in X$ such that for every neighborhood $O_x\subseteq X$ of $x$ there exists a function $f\in C(X)$ such that $\mu(f)\ne0$ and $\supp(f)\subseteq O_x$, where $\supp(f)=\overline{\{x\in X:f(x)\ne 0\}}$. This definition implies that the support $\supp(\mu)$ is a closed subset of $X$. Although the next lemma is proved in \cite[Lemma~3.2]{BG-JNP} we add its short proof to make the paper self-contained.

\begin{lemma}\label{l:3.1}
Let $X$ be a Tychonoff space, and let $\mathcal{S}$ be a directed set of functionally bounded sets in $X$. If a functional $\mu\in C(X)'$ is continuous in the topology $\Tau_{\mathcal S}$, then $\supp(\mu)\subseteq\overline{S}$ for some set $S\in\mathcal S$ such that $[S;0]\subseteq\mu^{-1}(0)$.
\end{lemma}

\begin{proof}
By the continuity of $\mu$ in the topology $\Tau_{\mathcal S}$, there exist a set $S\subseteq \mathcal S$ and $\e>0$  such that $\mu([S;\e])\subseteq(-1,1)$. Then $$[S;0]=\bigcap_{n\in\IN}[S;\tfrac\e{n}]\subseteq\bigcap_{n\in\w}\mu^{-1}\big((-\tfrac1n,\tfrac1n)\big)=\mu^{-1}(0).$$ It remains to prove that $\supp(\mu)\subseteq\overline S$. In the opposite case we can find a function $f\in C(X)$ such that $\mu(f)\ne 0$ and $\supp(f)\cap \overline S=\emptyset$. On the other hand, $f\in [S;0]\subseteq \mu^{-1}(0)$ and hence $\mu(f)=0$. This contradiction shows that $\supp(\mu)\subseteq\overline S$.
\end{proof}


\begin{lemma}\label{l:3.2}
Let $X$ be a Tychonoff space, and let $D$ be a dense subset of $X$. If a linear functional $\mu\in C(X)'$ is continuous in the topology $\TTT_{k{\restriction}D}$, then $\supp(\mu)$ is a compact subset of $D$ and $[\supp(\mu);0]\subseteq\mu^{-1}(0)$.
\end{lemma}

\begin{proof}
By the continuity of $\mu$ in the topology $\Tau_{k{\restriction}D}$, there exist a compact subset $K\subseteq D$ and $\e>0$ such that $\mu([K;\e])\subseteq(-1,1)$. By (the proof of) Lemma~\ref{l:3.1}, $\supp(\mu)\subseteq K$ and $[K;0]\subseteq\mu^{-1}(0)$. Since $\supp(\mu)$ is a closed subset of $X$, $\supp(\mu)$ is closed in $K$ and hence $\supp(\mu)$ is a compact subset of $D$.

It remains to prove that $[\supp(\mu);0]\subseteq\mu^{-1}(0)$. 
This can be done repeating the arguments of the proof of Lemma 3.3 in \cite{BG-JNP}, nevertheless we provide a detailed proof for the reader convenience. 
To derive a contradiction, assume that $[\supp(\mu);0]\not\subseteq\mu^{-1}(0)$ and hence there exists a continuous function $f\in C(X)$ such that $\mu(f)\ne 0$ but $f{\restriction}_{\supp(\mu)}=0$. Multiplying $f$ by a suitable constant, we can assume that $\mu(f)=2$. Embed the space $X$ into its Stone--\v{C}ech compactification $\beta X$. By the Tietze--Urysohn Theorem, there exists a continuous function $\bar f\in C(\beta X)$ such that $\bar f{\restriction}_K=f{\restriction}_K$. It follows from $[K;0]\subseteq\mu^{-1}(0)$ that $\mu(f)=\mu(\bar f{\restriction}_X)$.

Consider the open neighborhood $U=\{x\in \beta X:|\bar f(x)|<\e\}$ of $\supp(\mu)$ in $\beta X$. By the definition of support $\supp(\mu)$, every point $x\in K\setminus U$ has an open neighborhood $O_x\subseteq \beta X$ such that $\mu(g)=0$ for any function $g\in C(X)$ with $\supp(g)\subseteq O_x\cap X$. Observe that $U\cup\bigcup_{x\in K\setminus U}O_x$ is an open neighborhood of the compact set $K$ in $\beta X$. So there is a finite family $F\subseteq K\SM U$ such that $K\subseteq U\cup\bigcup_{x\in F} O_x$. Let $1_{\beta X}$ denote the constant function $\beta X\to\{1\}$. By the paracompactness of the compact space $\beta X$, there is a finite family $\{\lambda_0,\dots,\lambda_n\}$ of continuous functions $\lambda_i:\beta X \to[0,1]$ such that $\sum_{i=0}^n\lambda_i=1_{\beta X}$ and for every $i\in\{0,\dots,n\}$, the support $\supp(\lambda_i)$ is contained in some set $V\in\{\beta X\setminus K,U\}\cup\{O_x:x\in F\}$. We lose no generality assuming that
\[
\bigcup_{i=0}^j \supp(\lambda_i)\subseteq \beta X\setminus K, \quad \bigcup_{i=j+1}^s\supp(\lambda_i)\subseteq U,
\]
and for every $i\in\{s+1,\dots,n\}$ there exists $x_i\in F$ such that $\supp(\lambda_i)\subseteq O_{x_i}$.

Replacing the functions, $\lambda_0,\dots,\lambda_j$ by the single function $\sum_{i=0}^j\lambda_i$ and the functions, $\lambda_j+1,\dots,\lambda_s$ by the single function $\sum_{j+1}^s\lambda_i$, we can assume that $j=0$ and $s=1$. In this simplified case we have $\supp(\lambda_0)\subseteq \beta X\setminus K$, $\supp(\lambda_1)\subseteq U$ and $\supp(\lambda_i)\subseteq O_{x_i}$ for all $i\in\{2,\dots,n\}$.

For every $i\in n$, consider the function $f_i\in C(\beta X)$, defined by  $f_i(x):=\lambda_i(x)\cdot \bar f(x)$ for $x\in\beta X$. Then $\bar f=\sum_{i\in n} f_i$.

It follows from $\supp(f_0)\subseteq\supp(\lambda_0)\subseteq\beta X\setminus K$ and $[K;0]\subseteq \mu^{-1}(0)$ that  $f_0{\restriction}_K =0$ and $\mu(f_i{\restriction}_X)=0$.

Since $K\subseteq U$, $\bar f(U)\subseteq(-\e,\e)$ and $\supp(f_1)\subseteq \supp(\lambda_1)\subseteq U$, the function $f_1{\restriction}_X= ({\bar f}{\cdot}\lambda_1){\restriction}_X$ belongs to the set $[K;\e]\subseteq\mu^{-1}\big((-1,1)\big)$ and hence
\begin{equation}\label{eq:fle1}
\big|\mu\big(f_1{\restriction}_X\big)\big|\le 1.
\end{equation}

For every $i\in\{2,\dots,n\}$, we have $\supp(f_i)\subseteq\supp(\lambda_i)\subseteq O_{x_i}$ and hence $\mu(f_i{\restriction}_X)=0$ by the choice of $O_{x_i}$.
\smallskip

Now we see that
\[
2=\mu(f)=\mu(\bar f{\restriction}_X)=\mu(f_0{\restriction}_X)+\mu(f_1{\restriction}_X) +\mu\Big(\sum_{i=2}^n f_i{\restriction}_X\Big)=\mu(f_1{\restriction}_X),
\]
which contradicts (\ref{eq:fle1}).
\end{proof}

The following example shows that Lemma~\ref{l:3.2} is not true for linear  functionals which are continuous only in the topology $\Tau_b$.

\begin{example} {\em
Let $\w_1$ be the space of countable ordinals, endowed with the order topology. It is well-known that the space $\w_1$ is pseudocompact and every continous function $f:\w_1\to\IF$ is eventually constant and hence has $\lim\limits_{\alpha\to\w_1}f(\alpha)$. Then $\lim_{\w_1}:C(\w_1)\to\IF$, $\lim_{\w_1}:f\mapsto \lim\limits_{\alpha\to\w_1}f(\alpha)$, is a well-defined continuous linear functional on the Banach space $C_b(\w_1)$. It is easy to see that $\supp(\lim_{\w_1})=\emptyset$ and hence $[\supp(\lim_{\w_1});0]=[\emptyset;0]=C(\w_1)\not\subseteq\lim_{\w_1}^{-1}(0)$.\qed}
\end{example}

The next lemma follows from Lemma 3.4 from \cite{BG-JNP} because $\TTT_{\mathcal{S}}\subseteq \TTT_b$.

\begin{lemma}[\protect{\cite[Lemma~3.4]{BG-JNP}}]\label{l:3.3}
Let $X$ be a Tychonoff space, and let $\mathcal S$ be a directed family of functionally bounded sets in $X$. For any bounded subset $\mathcal M\subseteq (C_{\TTT_{\mathcal S}}(X))'_{w^\ast}$ the set $\supp(\M)=\bigcup_{\mu\in\mathcal M}\supp(\mu)$ is functionally bounded in $X$.
\end{lemma}

\begin{lemma}\label{l:almost-trivial}
Let $X$ be a dense subspace a Tychonoff space $K$ and $Z$ be a finite-dimensional linear subspace of $C(K)$. Then there exists a finite set $F\subseteq X$ such that the restriction operator $R:C(K)\to\IF^F$, $R:f\mapsto (f(x))_{x\in F}$, is injective on the set $Z$.
\end{lemma}

\begin{proof} The density of $X$ in $K$ implies that the restriction operator $R:C(K)\to\IF^X$ is injective. Let $e_1,\dots,e_n$ be a basis of the finite-dimensional linear space $Z$. Then the vectors $e_1,\dots,e_n$ are linearly independent in $C(K)$ and, by the injectivity of the operator $R$, the vectors $R(e_1),\dots,R(e_n)$ are linearly independent in $\IF^X$.
Consider the compact set $S=\{(x_1,\dots,x_n)\in \IF^n:\sum_{i=1}^n|x_i|=1\}$ in $\IF^n$ and the continuous map
\[
\Lambda_X:S\times (\IF^X)^n\to\IF^X,\quad \Lambda_X:((x_i)_{i=1}^n,(f_i)_{i=1}^n)\mapsto\sum_{i=1}^nx_if_i.
\]
The linear independence of the vectors $R(e_1),\dots,R(e_n)$ implies that  $0_X\notin \Lambda_X\big(S\times\big\{\big(R(e_i)\big)_{i=1}^n\big\}\big)$, where $0_X:X\to\{0\}\subseteq\IF$ is the zero of the linear space $\IF^X$. By the compactness of $S$, there exists a neighborhood $U$ of zero in $\IF^X$ such that $U\cap \Lambda\big(S\times\prod_{i=1}^n(R(e_i)+U)\big)=\emptyset$.
By the definition of the Tychonoff product topology on $\IF^X$, there exists a finite set $F\subseteq X$ such that $\pr_F^{-1}(0_F)\subseteq U$, where $\pr_F:\IF^X\to\IF^F$, $\pr_F:f\mapsto f{\restriction}_F$, is the projection map and $0_F$ is the zero of the linear space $\IF^F$. The choice of $F$ and $U$ ensure that the vectors $e_1{\restriction}_F,\dots,e_n{\restriction}_F$ are linearly independent in $\IF^F$ and hence the restriction operator $R:C(K)\to\IF^F$ is injective on the set $Z$.
\end{proof}

A subset $X$ of a topological space $M$ is called
\begin{itemize}
\item {\em $k$-dense} in $M$ if any compact set $K\subseteq M$ is contained in a compact set $S\subseteq M$ such that $S\cap X$ is dense in $S$;
\item {\em sequentially closed} in $M$ if for any sequence $\{x_n\}_{n\in\w}\subseteq X$ that converges in $M$ we have $\lim\limits_{n\to\infty}x_n\in X$.
\end{itemize}

The following theorem is the main result of this section.

\begin{theorem} \label{t:main3}
For a $\mu$-space $M$ and a dense subset $X$ in $M$, the following assertions are equivalent:
\begin{enumerate}
\item[{\rm(i)}] for every locally convex  topology $\TTT$ on $C(M)$ satisfying $\TTT_{p{\restriction}X}\subseteq \TTT\subseteq \TTT_{k{\restriction}X}$, the space $C_\TTT(M)$ is strongly Gelfand--Phillips;
\item[{\rm(ii)}] there is  a locally convex  topology $\TTT$ on $C(M)$ such that $\TTT_{p{\restriction}X}\subseteq \TTT\subseteq \TTT_{k{\restriction}X}$ and the space $C_\TTT(M)$ is strongly Gelfand--Phillips;
\item[{\rm(iii)}] $X$ contains a compact countable subspace $K$ of finite scattered height such that for every compact set $F\subseteq M$ the complement $F\cap X\setminus K$ is finite.
\end{enumerate}
If $X$ is $k$-dense and sequentially closed in $M$, then the conditions {\rm(i)--(iii)} are equivalent to
\begin{enumerate}
\item[{\rm(iv)}] for every locally convex topology $\TTT$ on $C(M)$ satisfying $\TTT_{p{\restriction}X}\subseteq \TTT\subseteq \TTT_{k}$, the space $C_\TTT(M)$ is strongly Gelfand--Phillips;
\item[{\rm(v)}] there is  a locally convex  topology $\TTT$ on $C(M)$ such that $\TTT_{p{\restriction}X}\subseteq \TTT\subseteq \TTT_k$ and the space $C_\TTT(M)$ is strongly Gelfand--Phillips.
\item[{\rm(vi)}] $X=M$ and $X$ contains a compact countable subspace $K$ of finite scattered height such that for every compact set $F\subseteq X$ the complement $F\setminus K$ is finite.
\end{enumerate}
If $X$ is $k$-dense in $M$, then the condition {\rm (v)} implies the condition
\begin{enumerate}
\item[{\rm(vii)}] $M$ contains a compact countable subspace $K$ of finite scattered height such that $M=X\cup K$ and for every compact set $ F\subseteq M$ the complement $F\setminus K$ is finite.
\end{enumerate}
If {\rm (iii)} holds, then a sequence $(\mu_n)_{n\in\w} \subseteq C_\TTT(X)'$ witnessing the  strong Gelfand--Phillips property of $C_\TTT(X)$ in {\rm(i)} and {\rm(iv)} can be chosen such that all $\mu_n$ have finite support $\supp(\mu_n)$ contained in $K$.
\end{theorem}

\begin{proof} The implications (i)$\Rightarrow$(ii), (iv)$\Rightarrow$(v), and (vi)$\Ra$(iii) are trivial.
\smallskip

(ii)$\Rightarrow$(iii) and (v)$\Rightarrow$(vi,vii). Assume that for some locally convex topology $\TTT$ on $C(M)$ with $\TTT_{p{\restriction}X}\subseteq\TTT\subseteq \TTT_k$, the locally convex space $C_\TTT(X)$ is strongly Gelfand--Phillips. We also assume that either $X$ is $k$-dense in $M$ or $\TTT\subseteq \TTT_{k{\restriction}X}$.

By the strong Gelfand--Phillips property of $C_\TTT(M)$, there exists a null sequence $(\mu_n)_{n\in\w}$ in $(C_\TTT(M))'_{w^\ast}$ such that $\|\mu_n\|_B\not\to 0$ for any barrel-bounded set $B\subseteq C_\TTT(M)$, which is not barrel-precompact. The continuity of the identity operator $\CC(M)\to C_{\TTT}(M)$ implies that the sequence $(\mu_n)_{n\in\w}$ is  null also in $\CC(M)'_{w^\ast}$.

 By Lemma~\ref{l:3.3}, the set $S:=\bigcup_{n\in\w}\supp(\mu_n)$ is functionally bounded and hence has compact closure $\bar S$ in the $\mu$-space $M$.

\begin{claim}\label{cl:finite}
For every compact set $F\subseteq M$ the complement $F\cap X\setminus \bar S$ is finite.
\end{claim}

\begin{proof} Replacing $F$ by the closure of $F\cap X$ in $M$, we can assume that $X\cap F$ is dense in $F$. Consider the bounded subset $B=\{f\in C(M):f{\restriction}_{\bar S}=0$ and $\|f\|_M\le 1\}$ of the locally convex space $C_k(M)$. Since $M$ is a $\mu$-space, the locally convex space $\CC(M)$ is barrelled by the Nachbin--Sirota Theorem \cite[11.7.5]{Jar}. Then $B$ is barrel-bounded in $\CC(M)$ and, by the continuity of the identity operator $\CC(M)\to C_\TTT(M)$, the set $B$ is barrel-bounded in $C_{\TTT}(M)$.  Applying Lemma~\ref{l:3.2} and taking into account that every function $f\in B$ is zero on the compact set $\bar S\supseteq \bigcup_{n\in\w}\supp(\mu_n)$, we conclude that $\|\mu_n\|_{B}=0$ for every $n\in\w$. Now the choice of the sequence $(\mu_n)_{n\in\w}$ ensures that the set $B$ is barrel-precompact in $C_\TTT(M)$. Observe that the set $[F;1]=[F\cap X;1]$ is a barrel in $C_{p{\restriction}X}(M)$ and hence a barrel in $C_\TTT(M)$. Since the set $B$ is barrel-precompact, for every $\e>0$ there exists a finite set $A_\e\subseteq B$ such that $B\subseteq A_\e +\e\cdot [F;1]$. Then for the restriction operator $T:C(M)\to C(F)$,
$T:f\mapsto f{\restriction}_F$, we have $T(B)\subseteq T(A_\e)+\e{\cdot} T([F;1])$. The Tietze--Urysohn Extension Theorem implies that $T([F;1])$ coincides with the closed unit ball  of the Banach space $C(F)$ and $T(B)$ coincides with the closed unit ball of the Banach subspace $L=\{f\in C(F):f{\restriction}_{F\cap \bar S}=0\}$ of $C(F)$. Now we see that the closed unit ball $T(B)$ of the Banach space $L$ is precompact in the Banach space $C(F)$ and hence the Banach space $L$ is finite-dimensional, which implies that the set $F\setminus \bar S$ is finite and so is the set $F\cap X\setminus\bar S\subseteq F\setminus\bar S$.
\end{proof}

\begin{claim}\label{cl:finite2}
If $X$ is $k$-dense in $M$, then $M=X\cup\bar S$.
\end{claim}

\begin{proof} By the $k$-density of $X$ in $M$,  every point $y\in M$ is contained in a compact set $F\subseteq M$ such that the set $F\cap X$ is dense in $K$. By Claim~\ref{cl:finite}, the set $F\cap X\setminus\bar S$ is finite and hence the set $(F\cap X)\cup\bar S$ is compact. Since the latter set is dense in $F\cup\bar S$, we obtain that $y\in F\cup\bar S=(F\cap X)\cup\bar S\subseteq X\cup\bar S$.
\end{proof}

If $\TTT\subseteq\TTT_{k{\restriction}X}$, then Lemma~\ref{l:3.2} implies that $S=\bigcup_{n\in\w}\supp(\mu_n)\subseteq X$ and hence $\bar S\cap X\supseteq S$ is dense in $\bar S$. In this case we put $K=\bar S$. If  $\TTT\not\subseteq\TTT_{k{\restriction}X}$, then $X$ is $k$-dense in $M$. In this case we can find a compact set $K\subseteq M$ such that $\bar S\subseteq K$ and $K\cap X$ is dense in $K$. In both cases we have a compact set $K\subseteq M$ such that $S\subseteq K$ and $K\cap X$ is dense in $K$. Since $\bar S\subseteq K$, Claim~\ref{cl:finite} ensures that $F\cap X\setminus K$ is finite for any compact set $F\subseteq M$. Moreover, if $X$ is $k$-dense in $M$, then $M=X\cup K$ by Claim~\ref{cl:finite2}.

It remains to prove that $K$ is countable  and has finite scattered height. We shall derive this from Theorem~\ref{t:Banach-C(K)-strong-GP} by showing that the Banach space $C(K)=\CC(K)$ is isomorphic to a subspace of the Banach space $c_0$.

Let $R:\CC(M)\to C(K)$, $R:f\mapsto f{\restriction}_K$, be the restriction operator. The Tietze--Urysohn Extension Theorem implies that this operator is surjective. For every $n\in\w$, the inclusion $\supp(\mu_n)\subseteq K$ and Lemma~\ref{l:3.2} imply the existence of a linear continuous functional $\lambda_n\in C(K)'$ such that $\mu_n=\lambda_n\circ R$. The surjectivity of the operator $R$ ensures that the sequence $(\lambda_n)_{n\in\w}$ is null in $C(K)'_{w^\ast}$. Consider the operator
\[
\Lambda:C(K)\to c_0, \quad \Lambda(f):=\big(\lambda_n(f)\big)_{n\in\w},
\]
which is continuous by the Uniform Boundedness Principle.

\begin{claim}
The kernel of the operator $\Lambda:C(K)\to c_0$ is finite-dimensional.
\end{claim}

\begin{proof} Let $Z=\{f\in C(K):\Lambda(f)=0\}$ be the kernel of $\Lambda$.
Since the locally convex space $\CC(M)$ is barrelled, the bounded subset $P=\{f\in \CC(M):\|f\|_X\le 1$ and $f{\restriction}_K\in Z\}$ of $\CC(M)$ is barrel-bounded in $\CC(M)$ and hence barrel-bounded in $C_\TTT(M)$ (as $\TTT\subseteq\TTT_k$). Observe that for every $f\in P$ we have $\mu_n(f)=\lambda_n(f{\restriction}_K)=0$ for all $n\in\w$.
Consequently, $\|\mu_n\|_P=0$ for all $n\in\w$. Now the choice of the sequence $(\mu_n)_{n\in\w}$ ensures that the barrel-bounded set $P$ is barrel-precompact in $C_\TTT(M)$. Observe that the set $[K;1]=[K\cap X;1]$ is a barrell in $C_{p{\restriction}X}(M)$ and hence in $C_\TTT(M)$. Since $P$ is barrel-precompact in $C_\TTT(M)$, for every $\e>0$ there exists a finite set $A_\e\subseteq P$ such that $P\subseteq A_\e+\e{\cdot} [K;1]$ and hence $R(P)\subseteq P(A_\e)+\e{\cdot} R([K;1])$. By the Tietze-Urysohn Extension Theorem, the set $R([K;1])$ coincides with the closed unit ball of the Banach space $C(K)$ and the set $R(P)$ coincides with the closed unit ball of the Banach subspace $Z$ of $C(K)$. Now we see that the closed unit ball $R(P)$ of the Banach space $Z$ is precompact in the Banach space $C(K)$, which implies that the space $Z$ is finite-dimensional.
\end{proof}

Applying Lemma~\ref{l:almost-trivial},  we can append to the sequence $(\mu_n)_{n\in\w}$ finitely many Dirac functionals supported at points of the set $K\cap X$, which separate points of the finite-dimensional space $Z$,  and assume that the operator $\Lambda:C(K)\to c_0$ is injective.

\begin{claim} \label{claim:sGP-1}
$\Lambda$ is an isomorphic embedding.
\end{claim}

\begin{proof}
In the opposite case, we could find a sequence of functions $F=\{f_n\}_{n\in\w}\subseteq C(K)$ such that $\|f_n\|_{K}=1$ for all $n\in\w$ and $\Lambda(f_n)\to 0$ in the Banach space $c_0$. Let $(e'_n)_{n\in\w}$ be the sequence of coordinate functionals in $c_0$. Since the set $F_0=\{0\}\cup\Lambda(F)$ is compact in $c_0$, Proposition \ref{p:precompact-c0} implies $\|e_n'\|_{F_0}\to 0$.

By the Tietze--Urysohn Theorem, for every $n\in\w$ there exists a function $\hat f_n\in C(M)$ such that $\hat f_n{\restriction}_{K}=f_n$ and $\|\hat f_n\|_{M}=1$. It follows that the set $\hat F=\{\hat f_n:n\in\w\}$ is bounded in $\CC(M)$. Since $M$ is a $\mu$-space, the locally convex space $\CC(M)$ is barreled. Then the bounded set $\hat F$ is barrel-bounded in $\CC(M)$ and also in $C_\TTT(M)$.
For every $n,m\in\w$, we have $\mu_n(\hat f_m)=\lambda_n(f_m)$ and hence
\[
\|\mu_n\|_{\hat F}=\sup_{m\in\w}|\mu_n(\hat f_m)|=\sup_{m\in\w}|\lambda_n(f_m)|=\|e'_n\|_{F_0}\to 0.
\]
The choice of the sequence $(\mu_n)_{n\in\w}$ implies that the barrel-bounded set $\hat F$ is barrel-precompact in $C_\Tau(M)$. Observe that the set $[K;1]=[K\cap X;1]$ is a barrel in $C_{p{\restriction}X}(M)$ and hence is $C_\TTT(M)$. For every $\e>0$, by the barrel-precompactness of $\hat F$, there exists a finite set $A_\e\subseteq \hat F$ such that $\hat F\subseteq A_\e+\e{\cdot} [K;1]$ and hence $F=R(\hat F)\subseteq R(A_\e)+\e{\cdot} R([K;1])$. The Tietze--Urysohn Theorem ensures that $R([K;1])$ coincides with the closed unit ball of the Banach space $C(K)$. Now we see that the set $F=R(\hat F)$ is precompact in the Banach space $C(K)$ and hence the sequence $(f_n)_{n\in\w}$ contains a subsequence $(f_{n_k})_{k\in\w}$ that converges to some function $f_\infty\in C(K)$ with $\|f_\infty\|_K=1$. On the other hand, $\Lambda(f_\infty)=\lim_{n\to\infty}\Lambda(f_{n_k})=0$, which contradicts the injectivity of $\Lambda$. This contradiction shows that the Banach space $C(K)$ is isomorphic to a subspace of $c_0$.
\end{proof}

By Theorem \ref{t:Banach-C(K)-strong-GP} and Claim \ref{claim:sGP-1}, the compact space $K$ is countable and has finite scattered height.
This completes the proof of the implication (v)$\Ra$(vii).
\smallskip

The following claim completes the proof of the implication (ii)$\Ra$(iii).  

\begin{claim}\label{cl:3.8}
If $\TTT\subseteq\TTT_{k{\restriction}X}$, then the function space $E=C_{k{\restriction}X\cap K}(K)$ is strongly Gelfand--Phillips and $K\subseteq X$.
\end{claim}

\begin{proof}
The continuity of the functionals $\mu_n$ in the topology $\TTT\subseteq\TTT_{k{\restriction}X}$ implies the continuity of the functionals $\lambda_n$ in the topology $\TTT_{k{\restriction}X\cap K}$ on the function space $C(K)$. Then $(\lambda_n)_{n\in\w}$ is a null sequence in the dual $E'_{w^\ast}$. We claim that $\|\lambda_n\|_B\not\to0$ for any barrel-bounded set $B\subseteq E$, which is not barrel-precompact. The continuity of the identity operator $C_{k{\restriction}X\cap K}(K)\to C_{p{\restriction}X\cap K}(K)$ implies that $B$ is barrel-bounded in the space $C_{p{\restriction}X\cap K}(K)$. The density of $K\cap X$ in $K$ implies that the barrel-bounded set $B$ is bounded in the Banach space $\CC(K)$. Since $B$ is not barrel-precompact in $E$, there exists a barrel $D$ in $E$ such that $B\not\subseteq A+D$ for every finite set $A\subseteq B$. Then we can choose a sequence of functions $\{f_n\}_{n\in\w}\subseteq B$ such that $f_n\notin f_k+D$ for every numbers $k<n$. By the Tietze--Urysohn Theorem, for every $n\in\w$ there exists a function $\hat f_n\in C(M)$ such that $\hat f_n{\restriction}_K=f_n$ and $\|\hat f_n\|_M=\|f_n\|_K$. Since $\sup_{n\in\w}\|\hat f_n\|_M=\sup_{n\in\w}\|f_n\|_K<\infty$, the set $\{\hat f_n\}_{n\in\w}$ is barrel-bounded in $\CC(M)$ and hence is barrel-bounded in $C_\TTT(M)$.

Since $D$ is a barrel in the Banach space $\CC(K)$, there exists $\e>0$ such that $[K;\e]\subseteq D$. It follows from $f_n\notin f_k+[K;\e]_{C(K)}$ that $\hat f_n\notin \hat f_k+[K;\e]_{C(M)}$ for every $k<n$. Since the set $K\cap X$ is dense in $K$, the set $[K;1]_{C(M)}=[K\cap X;1]_{C(M)}$ is a barrel in $C_{p{\restriction}X}(M)$ and hence a barrel in $C_\TTT(M)$. Now we see that the set $\hat B=\{\hat f_n:n\in\w\}$ is barrel-bounded and not barrel-precompact in $C_\TTT(M)$. The choice of the sequence $(\mu_n)_{n\in\w}$ ensures that $\|\mu_n\|_{\hat B}\not\to 0$. By Lemma~\ref{l:3.2}, the inclusion  $\supp(\mu_n)\subseteq K$ and  equality $\hat f_m{\restriction}_K=f_m$ imply that $\lambda_n(f_m)=\mu_n(\hat f_m)$ for any $n,m\in\w$. Then $\|\lambda_n\|_B\ge \|\mu_n\|_{\hat B}$ and hence $\|\lambda_n\|\not\to 0$, witnessing that the space $E$ is strongly Gelfand--Phillips.

Since $K\cap X$ is dense in $K$,  the span $L$ of $K\cap X$ in $E'$ is norming, and hence, by Lemma \ref{l:JNP-norming}, the space  $E$ is $\beta$-Banach with $E_\beta=\CC(K)$. By Theorem~\ref{t:strong-GP}, the dual space $E'$ is dense in the dual Banach space $\CC(K)'$, which is isometric to the Banach space $\ell_1(K)$  because $K$ is countable. On the other hand, by Lemma~\ref{l:3.2}, each functional $\mu\in E'$ has compact support in the set $X\cap K$, which implies that $E'\subseteq \ell_1(K\cap X)$. The density of $E'$ in $\CC(K)'$ ensures that $K\cap X=K$ and hence $K\subseteq X$.
\end{proof}

If the space $X$ is sequentially closed in $M$, then the dense subset $K\cap X$ of the metrizable compact space $K$ is closed in $K$ and hence $K=\overline{K\cap X}=K\cap X\subseteq X$. This completes the proof of the implication (v)$\Ra$(vi).
\smallskip

(iii)$\Rightarrow$(i,iv) Assume that $X$ contains a compact countable subspace $K$ of finite scattered height such that for every compact set $S\subseteq X$ the complement $S\cap X\setminus K$ is finite. Let $\TTT$ be any locally convex topology on $C(M)$ such that $\tau_{p{\restriction}X}\subseteq \TTT\subseteq\TTT_k$. If $X$ is not $\kappa$-dense in $M$, then we shall also assume that $\TTT\subseteq\TTT_{k{\restriction}X}$.

By Theorem \ref{t:Banach-C(K)-strong-GP}, the Banach space $C(K)$ is isomorphic to $c_0$. Observe that $C_p(K)'$ consists of finitely supported functionals and is dense in the dual Banach space $C(K)'$, which is isometric to $\ell_1(K)$. By Lemma \ref{l:JNP-norming} and Theorem~\ref{t:strong-GP}, the function space $C_p(K)$ is strongly Gelfand--Phillips. Consequently, there exists a null sequence $(\mu_n)_{n\in\w}\in C_p(K)'_{w^\ast}$ such that $\|\mu_n\|_B\not\to 0$ for any barrel-bounded subset $B\subseteq C_p(K)$ which is not barrel-precompact. The measures $\mu_n$ have finite support in $K\subseteq X\subseteq M$, and hence can be identified with finitely supported linear functionals on $C_{p{\restriction}X}(M)$. The continuity of the restriction map $C_{p{\restriction}X}(M)\to C_p(K)$ implies that $(\mu_n)_{n\in\w}$ is a null sequence in $C_{p{\restriction}X}(M)'_{w^\ast}$.

Since the identity operator $C_\TTT(M)\to C_{p{\restriction}X}(M)$ is continuous, the sequence $(\mu_n)_{n\in\w}$ remains null in $C_\TTT(M)'_{w^\ast}$. We claim that $\|\mu_n\|_{B}\not\to 0$ for any barrel-bounded subset $B\subseteq C_\TTT(M)$, which is not barrel-precompact. To derive a contradiction, assume that $\|\mu_n\|_{B}\to 0$. Since $B$ is not barrel-precompact, there exists a barrel $D\subseteq C_\TTT(M)$ such that $B\not\subseteq A+D$ for any finite set $A\subseteq B$. In this case we can inductively construct a sequence of functions $\{f_n\}_{n\in\w}\subseteq B$ such that $f_n\notin f_k+D$ for any numbers $k<n$.

By the continuity of the identity operator $\CC(M)\to C_\TTT(M)$, the set $D$ is also a barrel in $\CC(M)$. Since $M$ is a $\mu$-space, the space $\CC(M)$ is barelled according to the Nachbin--Shirota Theorem \cite[11.7.5]{Jar}. Then $D$ is a neighborhood of zero in $\CC(M)$ and hence $[S;\e]\subseteq D$ for some compact set $S\subseteq M$ and some $\e>0$. Replacing $S$ by $S\cup K$, we can assume that $K\subseteq S$.

\begin{claim}
$[S\cap X;\e]\subseteq D$.
\end{claim}

\begin{proof}
If $X$ is $k$-dense in $M$, then $S$ is contained in a compact set $\hat S\subseteq M$ such that $\hat S\cap X$ is dense in $\hat S$. On the other hand, the set $\hat S\cap X\setminus K$ is finite and hence the set $\hat S\cap X$ is compact and $S\subseteq \hat S=\overline{\hat S\cap X}=\hat S\cap X\subseteq X$. Then $[S\cap X;\e]=[S;\e]\subseteq D$.

So, assume that $X$ is not $k$-dense in $M$. In this case $\TTT\subseteq\TTT_{k{\restriction}X}$. Suppose for a contradiction that $[S\cap X;\e]\setminus D$ contains some function $f$.  Since the barrel $D$ is closed in the topology $\TTT\subseteq\TTT_{k{\restriction}X}$, there exist $\delta>0$ and a compact set $C\subseteq X$ such that $(f+[C;\delta])\cap D=\emptyset$. If follows from $f\in [S\cap X;\e]$ that $f(C\cap S)\subseteq f(X\cap S)\subseteq[-\e,\e]$. By the Tietze--Urysohn Theorem, there exists a continuous function $g:M\to [-\e,\e]$ such that $g(x)=f(x)$ for every $x\in C\cap S$. Define the function $\varphi:C\cup S\to\ff$ by the formula
\[
\varphi(x)=\begin{cases}
f(x)&\mbox{if $x\in C$};\\
g(x)&\mbox{if $x\in S$};
\end{cases}
\]
and observe that it is well-defined and continuous. By the Tietze--Urysohn Theorem, the function $\varphi$ can be extended to a bounded continuous function $\psi:M\to\ff$. Then $\psi\in (f+[C;\delta])\cap [S;\e]\subseteq (f+[C;\delta])\cap D$, which contradicts the choice of $C$ and $\delta$. This contradiction shows that $[S\cap X;\e]\subseteq D$.
\end{proof}

Since $[S\cap X;\e]\subseteq D$, we can replace the set $S$ by $\overline{S\cap X}$ and assume that $S\cap X$ is dense in $S$. Since the identity operator $C_\TTT(M)\to C_{p{\restriction}X}(M)$ is continuous, the set $B$ is barrel-bounded in $C_{p{\restriction}X}(M)$. Consider the restriction operator $R:C_{p{\restriction}X}(M)\to C_p(K)$ and observe that $R(B)$ is barrel-bounded in $C_p(K)$. By the barrelledness of the Banach space $C(K)$, the barrel-bounded subset $R(B)$ of $C_p(K)$ is bounded in $C(K)$. Since $\bigcup_{n\in\w}\supp(\mu_n)\subseteq K$, we obtain $\|\mu_n\|_{R(B)}= \|\mu_n\|_{B}\to 0$. The choice of the sequence $\{\mu_n\}_{n\in\w}\subseteq C(K)'_{w^\ast}$ ensures that the set $R(B)$ is barrel-precompact and hence $R(B)\subseteq R(P)+[K;\tfrac12\e]\subseteq C(K)$ for some finite set $P\subseteq B$. It follows that $B\subseteq P+[K;\tfrac{1}{2}\e]\subseteq C(M)$, and hence for every $n\in\w$ there exists a function $g_n\in P$ such that $f_n-g_n\in [K;\tfrac{1}{2}\e]$. Since the set $P$ is finite, for some $g\in P$ the set $\Omega=\{n\in\w:g_n=g\}$ is infinite. For any two numbers $n<m$ in $\Omega$ we have
\[
f_m-f_n=(f_m-g)-(f_n-g)\in [K;\tfrac{1}{2}\e]+[K;\tfrac{1}{2}\e]\subseteq [K;\e].
\]
It follows from $f_m-f_n\notin [S;\e]$ and $K\subseteq S$ that $f_m-f_n\notin [S\setminus K;\e]=[S\cap X\setminus K;\e]$. This means that the bounded subset $\{f_n{\restriction}_{S\setminus K}\}_{n\in\Omega}$ of the finite-dimensional space $C(S\cap X\setminus K)$ is not precompact, which is not possible. This contradiction shows that $\|\mu_n\|_B\not\to 0$. Therefore the sequence $(\mu_n)$  witnesses that the space $C_\TTT(M)$ is strongly Gelfand--Phillips. This completes the proof of the implications (iii)$\Ra$(i,iv).
%
\end{proof}

\begin{example} {\em
The compact space $M=\w+1=\w\cup\{\w\}$ endowed with the order topology and its $k$-dense subspace $X=\w$ satisfy the condition {\rm (vii)} of Theorem~\ref{t:main3} but not the conditions {\rm(iii)} and {\rm(vi)}. This example shows that {\rm (vii)} is not equivalent to any of the conditions {\rm (i)--(vi)}. Theorem~\ref{t:main3} implies that the locally convex space $C_{p{\restriction}\w}(\w+1)$ is not strongly Gelfand--Phillips in spite of the fact that the Banach space $C(\w+1)=\big(C_{p{\restriction}\w}(\w+1)\big)_{\!\beta}$ is isomorphic to $c_0$.\qed}
\end{example}

Theorem~\ref{t:main3} implies the following results.

\begin{theorem} \label{t:sGP-C}
For a Tychonoff space $X$ the following assertions are equivalent:
\begin{enumerate}
\item[{\rm(i)}] for every locally convex topology $\TTT$ on $C(X)$ satisfying $\TTT_p\subseteq \TTT\subseteq\TTT_k$, the space $C_\TTT(X)$ is strongly Gelfand--Phillips;
\item[{\rm(ii)}] there is  a locally convex vector topology $\TTT$ on $C(X)$ such that  $\TTT_p\subseteq \TTT\subseteq \TTT_k$ and the space $C_\TTT(X)$ is strongly Gelfand--Phillips;
\item[{\rm(iii)}] $X$ is a $\mu$-space containing a compact countable subspace $K\subseteq X$ of finite scattered height such that for every compact set $S\subseteq X$ the complement $S\setminus K$ is finite.
\end{enumerate}
If {\rm (i)--(iii)} hold, a sequence $(\mu_n)_{n\in\w} \subseteq C_\TTT(X)'$ witnessing the  strong Gelfand--Phillips property of $C_\TTT(X)$  can be chosen such that all $\mu_n$ have finite support contained in $K$.
\end{theorem}

\begin{proof}
The implication (i)$\Rightarrow$(ii) is trivial.
\smallskip

To prove that (ii)$\Ra$(iii), assume that $\Tau$ is a locally convex topology on $C(X)$ such that $\TTT_p\subseteq \TTT\subseteq\TTT_k$ and the lcs $C_\TTT(X)$ is strongly Gelfand--Phillips. By the classical result of Hewitt \cite[3.11.16]{Eng}, $X$ is a dense subspace of some $\mu$-space $M$ such that the restriction operator $C(M)\to C(X)$, $f\mapsto f{\restriction}_X$, is bijective. If $X$ is a $\mu$-space, then we can (and will) assume that $M=X$. Let $\tau$ be a unique locally convex topology on the function space $C(M)$ such that the restriction operator $C_\tau(M)\to C_\TTT(X)$ is a topological isomorphism. Then the spaces $C_{p{\restriction}X}(M), C_\tau(M)$ and $C_{k{\restriction}X}(M)$ are topologically isomorphic to the spaces $C_p(X)$, $C_\TTT(X)$ and $\CC(X)$, respectively. By Theorem~\ref{t:main3}, $X$ contains a compact countable set $K$ of finite scattered height such that for any compact subset $F\subseteq M$ the set $F\cap X\setminus K$ is finite. In particular, for any compact subset $S\subseteq X$, the complement $S\setminus K=S\cap {\color{red}X}\setminus K$ is finite. It remains to prove that $X$ is a $\mu$-space. As the restriction operator $C(M)\to C(X)$ is bijective, every functionally bounded set $F$ in $X$ is functionally bounded in $M$. Since $M$ is a $\mu$-space, $F$ has compact closure $\bar F$ in $M$. Since the set $\bar F\cap X\setminus  K$ is finite, the set $(\bar F\cap X)\cup K\subseteq X$ is compact. Now we see that the functionally bounded set $F$ is contained in the compact subset $(\bar F\cap X)\cup K$ of $X$, which means that $X$ is a $\mu$-space.
\smallskip

The implication (iii)$\Ra$(i) follows from Theorem~\ref{t:main3} applied to the $\mu$-space $M=X$ and its subset $X$.
\end{proof}

By Lemma \ref{l:Ck=Cb}, if $X$ is a $\mu$-space, then $\TTT_k=\TTT_b$.
This fact and Theorem \ref{t:sGP-C} immediately imply the next result.

\begin{corollary} \label{c:sGP-b}
For a $\mu$-space $X$, the following assertions are equivalent:
\begin{enumerate}
\item[{\rm(i)}] for every locally convex topology $\TTT$ on $C(X)$ satisfying $\TTT_p\subseteq \TTT\subseteq\TTT_b$, the space $C_\TTT(X)$ is strongly Gelfand--Phillips;
\item[{\rm(ii)}] there is  a locally convex vector topology $\TTT$ on $C(X)$ such that  $\TTT_p\subseteq \TTT\subseteq \TTT_b$ and the space $C_\TTT(X)$ is strongly Gelfand--Phillips;
\item[{\rm(iii)}] $X$ contains a compact countable subspace $K\subseteq X$ of finite scattered height such that for every compact set $S\subseteq X$ the complement $S\setminus K$ is finite.
\end{enumerate}
If {\rm (i)--(iii)} hold, a sequence $(\mu_n)_{n\in\w} \subseteq C_\TTT(X)'$ witnessing the  strong Gelfand--Phillips property of $C_\TTT(X)$  can be chosen such that all $\mu_n$ have finite support contained in $K$.
\end{corollary}

A topological space $X$ is defined to be {\em bounded-finite} if every functionally bounded set in $X$ is finite. It is clear that each bounded-finite space is a $\mu$-space. Corollary~\ref{c:sGP-b} implies the following helpful assertion.

\begin{corollary} \label{c:C0-sGP}
For any bounded-finite Tychonoff space $X$, the function space $C_p(X)=\CC(X)=C_b(X)$ is strongly Gelfand--Phillips.
\end{corollary}

In the next corollary of Theorem~\ref{t:main3} we characterize Tychonoff spaces $X$ for which the space of bounded functions $C^b(X)$ endowed with a locally convex vector topology $\TTT$ between $\TTT_p$ and $\TTT_n$ has the strong Gelfand--Phillips property, where $\TTT_n$ denotes the topology on $C^b(X)$ generated by the norm $\|\cdot\|_X$.

\begin{theorem}\label{t:Cb-sGP}
For any Tychonoff space $X$, the following assertions are equivalent:
\begin{enumerate}
\item[{\rm(i)}] for every locally convex topology $\TTT$ on $C^b(X)$ satisfying $\TTT_p\subseteq \TTT\subseteq \TTT_n$, the space $C_\TTT^b(X)$ is strongly Gelfand--Phillips;
\item[{\rm(ii)}] there is  a locally convex topology $\TTT$ on $C^b(X)$ such that $\TTT_p\subseteq \TTT\subseteq \TTT_n$ and the space $C_\TTT^b(X)$ is strongly Gelfand--Phillips property;
\item[{\rm(iii)}] $X$ is a  compact countable space of finite scattered height.
\end{enumerate}
\end{theorem}

\begin{proof}
The implication (i)$\Rightarrow$(ii) is trivial.
\smallskip

(ii)$\Rightarrow$(iii) Let $\TTT$ be a locally convex topology on $C^b(X)$ such that $\TTT_p\subseteq \TTT\subseteq \TTT_n$ and the space $C_\TTT^b(X)$ is strongly Gelfand--Phillips. Let  $M=\beta X$ be the Stone-\v Cech compactification of $X$. It is clear that the restriction map $R:C(M)\to C^b(X)$, $R:f\mapsto f{\restriction}_X$, is bijective. Let $\tau$ be the unique locally convex topology on $C(M)$ such that the restriction operator $R:C_\tau(M)\to C^b_\TTT(X)$ is a topological isomorphism. The inclusions $\TTT_p\subseteq \TTT\subseteq\TTT_n$ imply that $\TTT_{p{\restriction}X}\subseteq\tau\subseteq\TTT_k$.  Since the space $M=\beta X$ is compact, the dense subspace $X$ of $M$ is  $k$-dense in $M$. Then the implication (v)$\Ra$(vii) in Theorem~\ref{t:main3} yields a compact countable subspace $K\subseteq M$ of finite scattered height such that $M=X\cup K$ and for every compact set $S\subseteq M$ the complement  $S\setminus K\subseteq X$ is finite. In particular,  for $S=M$ we obtain that the set $M\SM K=X\SM K$ is finite and hence the compact space $M=(M\SM K)\cup K$ is countable and metrizable.  We claim that $X$ is closed in $M$. In the opposite case, we can find a point $x_\infty\in M\SM X$ and a one-to-one sequence $(x_n)_{n\in\w}\subseteq X$ converging to $x_\infty$. Therefore the set $D=\{x_n\}_{n\in\w}$ is closed and discrete in $X$. Since $X$ is metrizable, it is normal and hence, by Corollary 3.6.8 of \cite{Eng}, the closure $\overline{D}$ of $D$ in $M=\beta X$ is homeomorphic to $\beta D=\beta\w$ which is impossible since $\beta\w$ is not metrizable. This contradiction shows that $X$ is closed in $M$ and hence $M=X$ by the density of $X$ in $M$. Since $M\setminus K$ is finite, the space $M=X$ is compact, countable and has finite scattered height.
\smallskip

(iii)$\Rightarrow$(i) If $X$ is a countable compact space of  finite scattered height, then $X$ is a $\mu$-space, $C^b_\TTT(X)=C_\TTT(X)$ and the assertion follows from Theorem \ref{t:main3}.
\end{proof}

As an immediate corollary of Theorem~\ref{t:Cb-sGP}  we obtain the following complement of Theorem \ref{t:Banach-C(K)-strong-GP}.

\begin{corollary}\label{c:Cb-sGP}
For a pseudocompact space $X$, the following assertions are equivalent:
\begin{enumerate}
\item[{\rm(i)}] for every locally convex topology $\TTT$ on $C(X)$ satisfying $\TTT_{p}\subseteq \TTT\subseteq \TTT_b=\TTT_n$, the space $C_\TTT(X)$ has the strong Gelfand--Phillips property;
\item[{\rm(ii)}] there is  a locally convex  topology $\TTT$ on $C(X)$ such that $\TTT_{p}\subseteq \TTT\subseteq \TTT_b=\TTT_n$ and the space $C_\TTT(X)$ has the strong Gelfand--Phillips property;
\item[{\rm(iii)}] $X$ is a countable compact space of finite scattered height.
\end{enumerate}
\end{corollary}

\section{Examples} \label{sec:exa-sGP}


In this section we show that the class of all locally convex spaces with  the strong Gelfand--Phillips property behaves badly with respect to standard operations as taking completions, infinite products or (closed) subspaces. We also construct a Banach space that cannot be embedded into a strongly Gelfand--Phillips locally convex space.

We start from the operation of taking infinite power of a locally convex space. By Proposition \ref{p:finite-prod-sGP} the finite product of strongly Gelfand--Phillips spaces is strongly Gelfand--Phillips. But for infinite products this is not true in general as the next example shows. 

\begin{example} \label{exa:c-product-sGP}
The Banach space $c_0$ is strongly Gelfand--Phillips but its countable power $(c_0)^\w$ is not strongly Gelfand--Phillips.
\end{example}

\begin{proof} By Theorem~\ref{t:Banach-C(K)-strong-GP}, the Banach space $c_0$ is strongly Gelfand--Phillips. To show that $(c_0)^\w$ is not strongly Gelfand--Phillips, we recall first that the Banach space $c_0$ is isomorphic to the Banach space $C(K)$ for any countable compact space $K$ with a unique non-isolated point. Then the countable power $(c_0)^\w$ is isomorphic to the function space $C_k(K\times\w)$ where $\w$ is endowed with the discrete topology. By Theorem~\ref{t:sGP-C}, the function space $C_k(K\times \w)$ is not strongly Gelfand--Phillips.
\end{proof}

Let us recall that a locally convex space $E$ is {\em feral} if every bounded subset of $E$ is finite-dimensional. Observe that if a normed space $E$ is feral, then the closed unit ball of $E$ is finite-dimensional and hence $E$ is finite-dimensional as well. We shall consider a wider class of locally convex spaces introduced in \cite{BG-b-feral}:  a locally convex space $E$ is called {\em $b$-feral} if every barrel-bounded subset of $E$ is finite-dimensional. Clearly, every feral spaces is $b$-feral, but the converse is not true in general, see Example \ref{exa:b-feral-normed} below.
The next simple assertion is proved in  \cite{BG-b-feral}, we give its proof to make the paper more self-contained.

\begin{proposition}[\cite{BG-b-feral}] \label{p:semi-M-sGP}
Every $b$-feral space $E$ is strongly Gelfand--Phillips.
\end{proposition}

\begin{proof}
By definition every barrel-bounded subset of $E$ is finite-dimensional and hence is barrel-precompact. Then the constant sequence of zero functionals on $E$ witnesses that $E$ is strongly Gelfand--Phillips.
\end{proof}

Although the $b$-feralness of $E$ in the next example follows from a general assertion proved in \cite{BG-JNP}, we provide a simple direct proof for the reader convenience and to make the article self-contained.

\begin{example} \label{exa:b-feral-normed}
Let $E$ be the subspace of the Banach space $\ell_1$ consisting of all eventually zero sequences. The normed space $E$ is $b$-feral (but not feral) and hence it has the strong Gelfand--Phillips property. However, the completion $\ell_1$ of $E$ is not strongly Gelfand--Phillips.
\end{example}

\begin{proof}
Suppose for a contradiction that there is an infinite-dimensional barrel-bounded subset $A$ of $E$. So there are a strictly increasing sequence $(n_k)_{k\in\w}$ of natural numbers and a  sequence $\{x_k\}_{k\in\w}\subseteq A$ such that $\supp(x_k)\subseteq \{0,\dots, n_{k+1}-1\}$ and $a_{n_k}:=|x_k(n_k)|\not=0$. For every $n\in \w\SM\{n_k\}_{k\in\w}$ we put $a_n:=1$. Consider the set $U:=E\cap \prod_{n\in\w} \big[ -\tfrac{a_n}{n+1},\tfrac{a_n}{n+1}\big]$. It is clear that $U$ is a barrel in $E$. However, for every $\lambda\in\w$ there is $k\in\w$ such that $\lambda<n_k$ and hence $x_k\not\in \lambda U$. Thus $A$ is not barrel-bounded, a contradiction. Thus $E$ is $b$-feral. Being an infinite-dimensional normed space, $E$ cannot be feral.

By Proposition \ref{p:semi-M-sGP}, the space $E$ has the strong Gelfand--Phillips property. However, the completion $\ell_1$ of $E$ is not  strongly Gelfand--Phillips  by Theorem \ref{t:Banach-strong-GP} (or Theorem \ref{t:strong-GP}).
\end{proof}

Strongly Gelfand--Phillips spaces can contain dense subspaces without the strong Gelfand--Phillips property as the following example shows. In this example the ordinal $\w$ is endowed with the discrete topology.

\begin{example} \label{exa:dense-subspace-sGP}
The space $C_p(\w)=\mathbb{F}^\w$ is strongly Gelfand--Phillips  {\rm by  Corollary~\ref{c:C0-sGP}, but} its dense subspace $C_p^b(\w)$ is not strongly Gelfand--Phillips, {\rm by Theorem \ref{t:Cb-sGP}.}
\end{example}


A topological space $X$ is called {\em compact-finite} if every compact subspace of $X$ is finite. A map $f:Y\to Z$ is called {\em hereditarily quotient} if for any subspace $B\subseteq Z$ the map $f{\restriction}_B:B\to Z$ is quotient. By \cite[2.4.F(a)]{Eng}, a map $f:Y\to Z$ is hereditarily quotient if and only if for any $z\in Z$ and any open set $U\subseteq Y$ containing $f^{-1}(z)$ the image $f(U)$ is a neighborhood of $z$.

Let $X$ be a $T_1$-space. The {\em compact-finite resolution} of $X$ at a point $x\in X$ is the space
\[
\mathcal{R}_x(X)=\{x\}\cup\big((X\setminus\{x\})\times\w\big)
\]
endowed with the topology consisting of the sets $U\subseteq \mathcal{R}_x(X)$ such that if $x\in U$, then the set
\[
\{x\}\cup\big\{y\in X: (\{y\}\times\w)\setminus U \; \mbox{ is finite}\big\}
\]
is a neighborhood of $x$ in $X$. It follows that $\mathcal{R}_x(X)\setminus \{x\}$ is an open discrete subspace of $\mathcal{R}_x(X)$. The space $\mathcal{R}_x(X)$ is discrete if and only if $x$ is an isolated point of $X$. If the point $x$ is non-isolated in $X$, then $x$ is a unique non-isolated points of $\mathcal{R}_x(X)$. The definition of the topology on the space $\mathcal{R}_x(X)$ implies that the map $q_x:\mathcal{R}_x(X)\to X$ defined by
\[
q_x(z)=\begin{cases}
x&\mbox{if $z=x$};\\
y&\mbox{if $z=(y,n)\in (X\setminus\{x\})\times\w$};
\end{cases}
\]
is continuous and open at the point $x\in\mathcal{R}_x(X)$ (recall that a function $f:Y\to Z$ between topological spaces is called {\em open at} a point $y\in Y$ if for any neighborhood $O_y$ of the point $y$ in $Y$ the image $f(O_y)$ is a neighborhood of the point $f(y)$ in $Z$).
The {\em compact-finite resolution} of the space $X$ is the topological sum
\[
\mathcal{R}(X)=\bigoplus_{x\in X}\mathcal{R}_x(X)=\bigcup_{x\in X}\mathcal{R}_x(X)\times\{x\}.
\]

Example \ref{exa:closed-non-sGP} below shows that a  strongly Gelfand--Phillips space may contain a {\em closed} subspace which does not have the strong Gelfand--Phillips property. To construct this example we need the following theorem which is of independent interest.

\begin{theorem}\label{t:T1-quotient-resol}
Let $X$ be a $T_1$-space. Then
\begin{enumerate}
\item[{\rm(i)}] $\mathcal{R}(X)$ is a zero-dimensional, paracompact  and compact-finite $\mu$-space;
\item[{\rm(ii)}]  the map $q:\mathcal{R}(X)\to X$, $ q:(z,x)\mapsto q_x(z)$,  is  hereditarily quotient;
\item[{\rm(iii)}] $C_p(X)$ is a closed subspace of the space $C_p\big(\mathcal{R}(X)\big)$.
\end{enumerate}
\end{theorem}

\begin{proof}
(i) Since for every $x\in X$, the space $\mathcal{R}_x(X)$ has only one non-isolated point, it is zero-dimensional and paracompact. Therefore $\mathcal{R}(X)$ is zero-dimensional and paracompact as well. Being paracompact, $\mathcal{R}(X)$ is a $\mu$-space.

To show that $\mathcal{R}(X)$ is compact-finite, it suffices to check that the space $\mathcal{R}_x(X)$  is compact-finite for every $x\in X$. Fix a compact subset $K\subseteq \mathcal{R}_x(X)$. Replacing $K$ by $K\cup\{x\}$, we can assume that $x\in K$. Observe that for every $y\in X\SM\{x\}$ the subspace $\{y\}\times\w$ is closed and discrete in $\mathcal{R}_x(X)$ and hence $(\{y\}\times \w)\cap K$ is finite, being a closed discrete subspace of the compact space $K$. Now the definition of the topology on $\mathcal{R}_x(X)$ ensures that the set $U=\{x\}\cup(\mathcal{R}_x(X)\SM K)$ is a neighborhood of $x$. So, $x$ is an isolated point of $K$ and hence all points of the compact space $K$ are isolated, which implies that $K$ is finite.
\smallskip

(ii) To prove that  the map $q:\mathcal{R}(X)\to X$, $ q:(z,x)\mapsto q_x(z)$,  is  hereditarily quotient, fix an arbitrary $x\in X$ and take any open set $U\subseteq \mathcal{R}(X)$ containing the preimage $q^{-1}(x)$. We should prove that $q(U)$ is a neighborhood of $x$ in $X$. Since $(x,x)\in (\mathcal{R}_x(X)\times\{x\})\cap q^{-1}(x)\subseteq U$, there exists an open neighborhood $V\subseteq \mathcal{R}_x(X)$ of $x$ such that $V\times\{x\}\subseteq U$. Since the map $q_x:\mathcal{R}_x(X)\to X$ is open at $x$, the image $q(U)\supseteq q_x(V)$ is a neighborhood of $x$ in $X$.
\smallskip

(iii) It is well known that if $q: Y\to Z$ is a quotient map, then the adjoint map $q^\ast : C_p(Z)\to C_p(Y)$, $q^\ast:f\mapsto f\circ q$, is an embedding onto a closed subspace. This result and (ii) imply that  $C_p(X)$ is a closed subspace of the space $C_p\big(\mathcal{R}(X)\big)$.
\end{proof}

\begin{example} \label{exa:closed-non-sGP}
Let $K$ be a countable compact space of infinite scattered height. Then:
\begin{enumerate}
\item[{\rm(i)}] $C_p(K)$ is not strongly Gelfand--Phillips;
\item[{\rm(ii)}] $C_p\big(\mathcal{R}(K)\big)$ is a separable metrizable strongly Gelfand--Phillips space;
\item[{\rm(iii)}] $C_p(K)$ is a closed subspace of the space $C_p\big(\mathcal{R}(K)\big)$.
\end{enumerate}
\end{example}

\begin{proof}
The assertion (i) follows from Theorem \ref{t:Banach-C(K)-strong-GP}, and (iii) follows from Theorem \ref{t:T1-quotient-resol}(iii).

(ii) Since $K$ is countable, the space $\mathcal{R}(K)$ is countable as well. Therefore $C_p\big(\mathcal{R}(K)\big)$ is a separable metrizable space. It is strongly Gelfand--Phillips by (i) of Theorem \ref{t:T1-quotient-resol} and Corollary \ref{c:C0-sGP}.
\end{proof}

\begin{proposition}\label{p:4.7}
Every locally convex space $E$ endowed with the weak topology is topologically isomorphic to a closed subspace of some strongly Gelfand--Phillips barrelled locally convex space.
\end{proposition}

\begin{proof} Let $E'$ be the dual space of $E$, endowed with the weak$^*$ topology. Observe that the map $\delta:E_w\to C_p(E')$ assigning to every $x\in X$ the function $\delta_x:E'\to \IF$, $\delta_x:f\mapsto f(x)$, is an isomorphic topological embedding. By  Theorem~\ref{t:T1-quotient-resol}, the compact-finite resolution $\mathcal R(E')$ of $E'$ is a compact-finite $\mu$-space and the space $C_p(E')$ is isomorphic to a closed linear subspace of $C_p(\mathcal R(E'))$. By Corollary~\ref{c:C0-sGP}, the function space $C_p(\mathcal R(E'))$ is strongly Gelfand-Phillips. Since $\mathcal R(E')$ is a compact-finite $\mu$-space, its function space $C_k(\mathcal R(E'))=C_p(\mathcal R(E'))$ is barrelled by the Nachbin--Sirota Theorem. Therefore, the space $E_w$ is topologically isomorphic to a subspace of the strongly Gelfand--Phillips barreled space $C_k(\mathcal R(E'))=C_p(\mathcal R(E'))$.
\end{proof}

Example \ref{exa:closed-non-sGP} and Proposition~\ref{p:4.7} motivate the question of whether every locally convex space embeds into a locally convex space with the strong Gelfand--Phillips property. This question has a negative answer as the following example shows. Recall that a locally convex space $E$ is {\em injective} in the class of locally convex spaces if $E$ is complemented in every locally convex space that contains $E$ as a subspace. It is well known that the Banach space $\ell_\infty$ is injective, see Theorems 10.1.2 and 10.1.3 in \cite{NaB}.

\begin{example} \label{exa:l-infty-not-sGP}
The Banach space $\ell_\infty$ does not embed into a strongly Gelfand--Phillips space.
\end{example}

\begin{proof}
Suppose for a contradiction that $\ell_\infty$ embeds into a strongly Gelfand--Phillips space $E$.  Since $\ell_\infty$ is injective, $\ell_\infty$ is complemented in $E$. Then, by Propositions \ref{p:beta-emb}(i) and \ref{p:sGP-hereditary}(ii), $\ell_\infty=C(\beta \w)$ is strongly Gelfand--Phillips. But this contradicts Theorem \ref{t:Banach-strong-GP}.
\end{proof}



\end{document}